\documentclass{article}

\usepackage[utf8]{inputenc}
\usepackage{amsmath,amsthm}
\usepackage{amssymb}
\usepackage[english]{babel}
\usepackage{mathtools}
\usepackage{verbatim}
\usepackage{parskip}
\usepackage{float}
\usepackage{rotating}
\usepackage{tikz}
\usetikzlibrary{decorations.markings}
\usepackage{authblk}
\usepackage{adjustbox}
\usepackage{ytableau}
\usetikzlibrary{matrix, calc, arrows}
\usepackage{comment}
\usepackage{fullpage}
\usepackage{upgreek}
\usepackage{xfrac}
\usepackage[shortlabels]{enumitem}
\usepackage[bookmarksopen=false,pdftex=true,breaklinks=true,%
      backref=page,pagebackref=true,plainpages=false,%
      hyperindex=true,pdfstartview=FitH,colorlinks=true,%
      pdfpagelabels=true,colorlinks=true,linkcolor=blue,%
      citecolor=red,urlcolor=green,hypertexnames=false%
      ]{hyperref}

\usepackage{graphicx}

\DeclareMathOperator{\lcm}{lcm}

\newtheorem{theorem}{Theorem}[section]
\newtheorem{corollary}{Corollary}[theorem]
\newtheorem{lemma}[theorem]{Lemma}
\newtheorem{definition}{Definition}
\newtheorem{conjecture}{Conjecture}

\newtheorem{example}[theorem]{Example}

\newtheorem{proposition}[theorem]{Proposition}
\usepackage[margin=1.1in]{geometry}
\usepackage{setspace}
\newcommand{\undb}{\underbrace}
\DeclareMathOperator{\per}{Per}
\DeclareMathOperator{\preper}{PrePer}
\newcommand{\Z}{\mathbb{Z}}
\newcommand{\N}{\mathbb{N}}
\newcommand{\bigO}{\mathcal{O}}
\newcommand{\ang}[1]{\left\langle #1\right\rangle}
\newcommand{\ceil}[1]{\left\lceil #1 \right\rceil}
\newcommand{\floor}[1]{\left\lfloor #1 \right\rfloor}
\newcommand{\abs}[1]{\left\lvert #1 \right\rvert}

\newcommand{\nth}{^{\text{th}}}
\newcommand{\inv}{^{\text{-}1}}

\newcommand{\biject}{\mathrel{\mathrlap{\hookrightarrow}}\mathrel{\mkern2mu\twoheadrightarrow}}

\newcommand{\ncount}{N}
\newcommand{\idk}{\beta}

\title{Superpolynomial period lengths of the winning positions in the subtraction game}

\author[1,2]{Istv\'an Mikl\'os} 
\author[3,4]{Logan Post} 

\affil[$1$]{R\'enyi institute\\ 1053 Budapest, Re\'altanoda u. 13-15\\ Hungary}

\affil[$2$]{SZTAKI\\ 1111 Budapest, L\'agym\'anyosi u. 11\\ Hungary}

\affil[$3$]{Budapest Semesters in Mathematics\\ 1071 Budapest, Bethlen G. t\'er 2\\ Hungary}

\affil[$4$]{Georgia Institute of Technology\\ 686 Cherry St NW, Atlanta, GA \\ USA}
\date{}

\begin{document}
\maketitle

\begin{abstract}
    Given a finite set of positive integers, $A$, and starting with a heap of $n$ chips, Alice and Bob alternate turns and on each turn a player chooses $x\in A$ with $x$ smaller or equal than the  current number of chips and subtract $x$ chips from the heap. The game terminates when the current number of chips becomes smaller than $\min\{A\}$ and no moves are possible. The player who makes the last move is the winner.  We can define $w^A(n)$ to be $1$ if Alice has a winning strategy with a starting heap of $n$ chips and $0$ if Bob has a winning strategy. By the Pigeonhole Principle, $w^A(n)$ becomes periodic, and it is easy to see that the period length is at most an exponential function of $\max\{A\}$. The typical period length is a linear function of $\max\{A\}$, and it is a long time open question if exponential period length is possible.

    We consider a slight modification of this game by introducing an initial seed $S$ that tells for the few initial numbers of chips whether the current or the opposite player is the winner. In this paper we show that the initial seed cannot change the period length of $w^A(n)$ if the size of $A$ is $1$ or $2$, but it can change the period length with $|A|\ge 3$. Further, we exhibit a class of sets $A$ of size $3$ and corresponding initial seeds such that the period length becomes a superpolynomial function of $\max\{A\}$.
\end{abstract}

\section{Introduction}\label{sec:introduction}

Game Theory is the theory of interactive situations or games among rational decision-makers or players in which the decisions of each player are contingent on the decisions of the others. Combinatorial Game Theory considers games with perfect information and without elements of chance. That is, at all times during the game, players have perfect information about the state of the game, and further, the moves in the game are entirely decided by the players, there is no elements of chance once the game has begun. We further require that a combinatorial game must end with a clear winner.

An example of a two-player combinatorial game is the subtraction game.
For a finite set $A\subset\N_+$, the \emph{$A$-subtraction game} is a two-player combinatorial game which proceeds as follows. We begin with heap of $n$-chips. Players Alice and Bob alternate turns, and on each turn a player chooses $x\in A$ with $x\leq n$, and subtracts $x$ chips from the heap, leaving $n-x$ remaining. The game terminates when $n<\min(A)$ and thus no moves are possible. The player who makes the last move is the winner. In general we consider a fixed $A$ and ask for which values of $n$ each player has a winning strategy. Note that when Alice makes a move $x$, Bob and Alice switch roles and we reduce to the $n-x$ game.

The subtraction game is also in the large class of combinatorial games called \emph{impartial games}. An impartial game is a combinatorial game in which the allowable moves depend only on the position and not on which of the two players is currently moving, and where the payoffs are symmetric. It is also in \emph{normal mode}, meaning that the winner is who can make the last possible move.
The Sprague-Grundy theorem \cite{Sprague,Grundy} says that any impartial game in normal mode is equivalent with a Nim game, which is the disjunctive sum of $\N_+$-subtraction games. Despite this reduction, we know little about the patterns of the winning positions of the subtraction game.

For any finite set $A\subset\N_+$, a dynamic programming recursion can compute which player has the winning strategy starting with a pile of size $n$. Simple reasoning by Pigeonhole Principle shows that the pattern of winning positions will eventually become periodic as $n$ takes all possible positive integers, and the period length cannot be longer than $2^{\max(A)}$. It is a long time open question if exponential period lengths exist in the substraction game. Alth\"ofer and B\"ultermann conjectured that superpolynomial period lengths might exist if $|A| \ge 5$ \cite{superlinear}.

For some $A$, the subtraction game has a pre-period in the winning positions before becomeing periodic, that is, a pattern of winning positions for small $n$ that is never repeated. We also know little about for which $A$ the subtraction game has a pre-period and for which $A$ it is \emph{purely periodic}, that is, has no pre-period.

In this paper, we generalize the subtraction game by modifying who is the winner for some small $n$. We call this initial pattern a \emph{seed}. We give a complete analysis of subtraction games with seeds for $|A| =1$ and $2$. We prove that for all seeds the game has no pre-period if $|A| = 1$ or $2$. For $A =\{a\}$, the period length is $2a$. For $A =\{a,b\}$, the period length can be any divisor of $a+b$, with a few exceptions. If $a$ and $b$ are relatively prime then the period can be any divisor except $1$, $4$, and $6$. We also compute the number of possible distinct period lengths over all seeds.
When $|A| =3$, then there might be pre-periods. We give characterisations of period and pre-period lengths for a large class of possible sets $A$. Finally, we show that superpolynomial period lengths exist already when $|A| =3$.

%
%

\section{Preliminaries}\label{sec:preliminaries}



\begin{definition} Let $w^A(n)$ be the winning indicator function of $A$, so $w^A(n)=1$ if Alice (the first player) has a winning strategy and $w^A(n)=0$ if Bob has a winning strategy.
\end{definition}
For brevity, we may use $w(n)$ to refer to $w^A(n)$. By definition, $w$ satisfies the recurrence relation
\begin{equation}\label{eq:recurrence}
w(n)=\begin{cases}
1&w(n-x)=0\text{ for some }x\in A\\
0&w(n-x)=1\text{ for all }x\in A
\end{cases}=1-\min\{w(n-x)\mid x\in A\}.
\end{equation}
This follows from the observation that for a particular $n$, Alice is in a winning position if she can subtract some $x$ to give Bob a losing position. Otherwise, she will certainly move to a winning position for Bob, and he will win. We can then also describe $w^A$ as the lexicographically least sequence in $n$ such that for all $x\in A$, $w(n)=0\implies w(n+x)=1$.

Note that it is natural to define $w(0):=0$, because if a player has previously made a move to $0$, then the next player will lose. Therefore, to satisfy recurrence relation \eqref{eq:recurrence} we shall define $w(n):=1$ for all $n<0$.

To describe an entire sequence $\{w^A(n)\}_{n=0}^\infty$ of winning positions, abbreviated as $\{w^A\}$, we use exponents to denote repeated values. A noted example in \cite[p. 86]{winningways1} is the set $\{2,4,7\}$. We find that $\{w^{\{2,4,7\}}\}= 0,0,1,1,1,1,0,1,1,0,\ldots$, which can be abbreviated to $\{w^{\{2,4,7\}}\}=0^21^401^20\ldots =0^21^2(1^20)^\infty$. In this example we have $w(0)=0$, $w(1)=0$, and $w(2)=1$. This follows from the rule that if $n<2$, Alice is unable to move, but at $n=2$, Alice may subtract 2 chips and win the game. Similarly $w(6)=0$ because for any move Alice makes, Bob can respond with a winning move. In general, if we present a prefix of $\{w^A\}$ of length $\ell$, we use ``$\underset{\ell}{\ldots}$" to indicate the continuation of the sequence and clarify the prefix's length for the reader. For example, $\{w^{\{2,4,7\}}\}=0^21^4\underset{6}{\ldots}$ may indicate that the first $6$ values of $w$ are given and the rest are not yet derived. Throughout the paper, we refer to $\alpha=\max(A)$. The following example is an easy generalization of one in \cite[p. 103]{winningways1}.

\begin{example}\label{ex:1tok}
Suppose $A=\{1,2,\ldots,\alpha\}$. Then $w^A(n)=\begin{cases}
0&(\alpha+1)\mid n\\1&\text{otherwise}\end{cases}$, so $\{w^A\}=\big(\,01^\alpha\,\big)^\infty$.
\end{example}
\begin{proof}
Suppose $n=k(\alpha+1)$. We claim Bob has a winning strategy. If Alice subtracts $x$, then Bob can subtract $\alpha+1-x$, reducing to $(k-1)(\alpha+1)$ chips. Bob can repeat this until there are $(0)(\alpha+1)$ chips, winning the game. Suppose $n=k(\alpha+1)+y$. Then Alice has a winning strategy. She may $y$ chips, reducing the game to $k(\alpha+1)$, then play as Bob would by countering each of his moves $x$ with $\alpha+1-x$.
\end{proof}

We define a notation for repeated concatenation of strings. By analogy to addition, for strings $w_1,w_2,\ldots w_k$, let
\newcommand{\conc}[2]{\underset{#1} {\overset{#2}{\boldsymbol{\LARGE||}}}}
\[
\sum_{i=1}^{k}w_i=w_1\circ w_2\circ \ldots \circ w_k.
\]
This notation satisfies the equality $\abs{\sum_{i=1}^{k}w_i}=\sum_{i=1}^{k}\abs{w_i}$. Recall that sequence concatenation is not a commutative operation, although the usual summation of numbers is one. When we use a summation symbol for a concatenation of strings, we will always use an index that defines the order of the concatenation.


\begin{definition}\label{def:period} A sequence $\{w^A(n)\}$ is \emph{periodic} over $p$ if there is some $N\in \N$ such that for all $n\geq N$, $w^A(n)=w^A(n+p)$. We say the \emph{period} of $w^A(n)$ is the least such $p$ and the \emph{preperiod} is the least such $N$. We denote these by $\per(A)$ and $\preper(A)$ respectively.
\end{definition}
In Example \ref{ex:1tok} we have $\per(\{1,\ldots,\alpha\})=\alpha+1$ and $\preper(\{1,\ldots, \alpha\})=0$. We also observe that $\per(\{2,4,7\})=3$ and $\preper(\{2,4,7\}=4$, because for all $n\geq 4$ it holds that $w(n)=w(n+3)$.

\begin{lemma}\label{lem:periodic} For any finite set $A\subset \N_+$, $\{w^A(n)\}$ is periodic.
\end{lemma}
\begin{proof} To prove this fact, we define a new tool called the vector of previous values. Given $A\subset \N_+$, let
\begin{equation}\label{eq:prevvalues}
\vec v^A(n):= \ang{w(n-\alpha),\ldots,w(n-2),w(n-1)}=\ang{v_1,\ldots,v_\alpha}\in \Z_2^\alpha.
\end{equation}
As shown above, $\vec v(n)$ will be an element of $\Z_2^\alpha$. Next, we use the recurrence relation to define a function $\mathcal F:\Z_2^\alpha\to\Z_2^\alpha$.
\begin{equation}\label{eq:vectorfunction}
\mathcal{F}(\vec v):=\ang{v_2,v_3,\ldots,v_{\alpha-1},1-\min\{v_{\alpha-x}\mid x\in A\}}\end{equation}
Thus by Recurrence \ref{eq:recurrence}, we have $\vec v(n+1)=\mathcal F\big(\vec v(n)\big)$. Note that $\Z_2^\alpha$ has $2^\alpha$ elements, so by the Pigeonhole Principle there must be distinct integers $0\leq N<M\leq 2^\alpha$ such that $\vec v(N)=\vec v(M)$. Therefore for all $n\geq N$, we have the equality 
\[\vec v(n)=\mathcal F^{n-N}\big(\vec v(N)\big)=F^{n-N}\left(\vec v(M)\right)=\vec v(n+M-N).\]
Therefore a single repeated vector guarantees periodicity of length $M-N$. 
\end{proof}
Indeed this Lemma often fails for infinite set games.
\begin{proposition}\label{prop:notperiodic}
For some $k\geq 2$, suppose $A=\{n^k\mid n\in \N\}$. Then the sequence $w^A$ is not periodic.\footnote{For $k=2$, the losing positions of this sequence are in the Online Encyclopedia of Integer Sequences \cite[A030193]{oeis}} 
\end{proposition}
\begin{proof}
We first show there must be infinitely many losing positions by contradiction. Suppose there are $\ell$ losing positions. Among all integers $\leq (2\ell)^k$, there are $2\ell$ elements of $A$. Because each winning position can be expressed as a losing position plus some $x\in A$, there can be combinatorially at most $2\ell^2$ distinct winning positions $\leq (2\ell)^k$. This implies $\ell+2\ell^2\geq (2\ell)^k$, which contradicts finiteness of losing positions. Thus $\{w^A\}$ has infinitely many zeros. If we suppose $w$ is periodic after some $N$ with period $p$, then there is some $n>N$ with $w(n)=0$. This implies that after $p^{k-1}$ periods, we will have $w(n+p^k)=0$, contradicting $p^k\in A$.
\end{proof}

The proof of Lemma \ref{lem:periodic} gives the result that for any finite $A$, $\preper(A)+\per(A)\leq 2^\alpha$. We can get a slightly tighter bound for dense sets, but no less than exponential in $\alpha$.
\begin{theorem}\label{thm:periodbound} 
For $A=\{a_1,a_2,\ldots,\alpha',\alpha\}$, let $\beta=\min(\alpha,a_1+\alpha')$. Then 
\begin{equation}\label{eq:universalperiodbound}
\preper(A)+\per(A)\leq (\beta+1)2^{\alpha-|A|}.
\end{equation}
\end{theorem}
\begin{proof}
First, we show that there cannot be a string of ones longer than $\beta$ in $\{w^A\}$. If some $w(n)$ is preceded by a string of $\alpha$ ones, then $w(n-x)=1$ for all $x\in A$, so $w(n)=0$. Alternately, if $n$ is preceded by a string of $a_1+\alpha'$ ones, this implies that in particular $w(n-a_1)=1$, and $(n-a_1)$ is preceded by $\alpha'$ ones. This means that for all $a_i\in A$ \textit{except for} $\alpha$, we have $a_i\leq \alpha'$ and therefore $w(n-a_1-a_i)=1$. Because $w(n-a_1)=1$, by process of elimination we conclude that $w(n-a_1-\alpha)=0$. Therefore $w(n-\alpha)=1$. Thus $w(n-x)=1$ for all $x\in A$, so $w(n)=0$. Hence $\beta$ bounds the number of consecutive ones in $\{w^A\}$.

Now, let $n_i$ be the $i^{th}$ zero in $\{w^A\}$. By the proof above $n_i-n_{i-1}\leq \beta+1$, so $n_i\leq i(\beta+1)$. In order to have $w(n_i)=0$, we require that within the vector $\vec v(n_i)$, the ${\alpha-x}\nth$ entry $\vec v(n_i)_{\alpha-x}=1$ for all $x\in A$. This fixes $|A|$ entries, so there are $2^{\alpha-|A|}$ possibilities for $\vec v(n_i)$. By the Pigeonhole Principle, there is some repeated vector $\vec v_{n_N}=\vec v_{n_M}$ for some $N<M\leq 2^{\alpha-|A|}$, and $n_M\leq (\beta+1)2^{\alpha-|A|}$. This proves the claim.
\end{proof}
A similar argument can give a bound of $(|A|+1)2^{\alpha-|A|}$. Using the approach from Proposition \ref{prop:notperiodic}, we can see that there are at most $(2^{\alpha-|A|}-1)\cdot(|A|)$ ways to express a winning position less than $n_{2^{\alpha-|A|}}$, which yields this improved constant.

\subsection{Initial Seed}\label{sec:seed}
We can generalize the subtraction game by changing the end state of the game. In \cite[(ix)]{superlinear}, Alth\"ofer and B\"ulterman suggest that this variant ``may be interpreted as simulations of certain computing devices," and pose open questions about this game. Through the analysis in Sections \ref{sec:ab}, \ref{sec:1,b,c}, and \ref{sec:superpolynomial} we find that this generalization also provides insights into the original game. We begin with two motivating examples.

\begin{example}\label{ex:misere} The \textit{\emph{Miser\'e}} mode of the $A$-subtraction game is the same game except the player to make the last move is the loser.
\end{example}
\begin{example}\label{ex:greedy} The \textit{\emph{Greedy}} mode of the $A$-subtraction game is the same game except a player may take $x>n$ chips, thereby making the heap negative. The game concludes when then heap is negative, and the player to make the last move is the winner. This is close to the version studied in \cite{superlinear}.\footnote{The exact game studied in \cite{superlinear} is the same but has $w(0):=0$; this cannot be generated from a seed as there is no function $w:\Z\to \{0,1\}$ satisfying recurrence \ref{eq:recurrence} for all $n\in \N$. For example, if $A=\{1,3\}$, we desire $w(0)=0$, $w(1)=1$, and $w(2)=1$. Any choice of $w(-1)$ leads to a contradiction.}
\end{example}
Notice that for $n\geq \alpha$, the recurrence relation for both of these games is the same as Equation \eqref{eq:recurrence}, but the ultimate sequence of winning positions may be different because of the players' final goals. We can account for this by adjusting negative values of $w(n)$, then allowing the recurrence relation to proceed for $n\geq 0$. 
\begin{definition} Define the \emph{seed} $S$ of a game to be the value of $\vec v(0)$, determining $w(n)$ for $n\in [-\alpha,0)$. The recurrence relation follows, so $\vec v(n):=\mathcal{F}^n(S)$, with $\mathcal F$ defined in Equation \eqref{eq:vectorfunction}. We define $\{w^{A,S}(n)\}_{n=0}^\infty$ to be the sequence generated by seed $S$. Further define $\per(A,S)$ and $\preper(A,S)$ to be the period and preperiod of $\{w^{A,S}\}$.
\end{definition}
For ease of notation, we interpret $S$ more generally as the negative values of $w^A(n)$. Ordinarily $|S|\leq\alpha$, and if not then $\vec v^{A,S}(0)$ is taken to be the last $\alpha$ elements of $S$. Similarly, if $|S|<\alpha$, then we presume $S$ to be preceded by infinitely many $1$'s, so let $w^{\{A,S\}}:=w^{\{A,1^{\alpha-|S|}S\}}$. It follows from this convention and Definition \ref{def:period} that $\per(A)=\per(A,1^\alpha)$, and thus in general we refer to a game with seed $1^\alpha$ as having `no seed'.

In Example \ref{ex:misere}, we observe that $S=0^{\min(A)}1^{\alpha-\min(A)}$ generates the sequence of winning positions for the Miser\'e mode of the subtraction game, and in Example \ref{ex:greedy}, we see that $S=0^\alpha$ generates the sequence for the Greedy game. By considering all seeds in $\Z_2^\alpha$ we describe a larger class of games. Some seeds generate games which are similar or identical to the original, while some are dramatically different. For example, the Miser\'e and Greedy modes cause only a translation in $w(n)$ by $\min(A)$ and $\alpha$ respectively.

Notice that Results \ref{lem:periodic} - \ref{thm:periodbound} consider the recurrence in generality and therefore hold for all seeds. In order to characterize $w^{A,S}$ over all seeds, we provide the following notation for the set of all  winning sequences and their periods.
\begin{align}\label{eq:defW,P}
\mathcal{W}^A&:=\Big\{\{w^{A,S}(n)\}_{n=0}^\infty\mid S\in \{0,1\}^\alpha\Big\}\\
\mathcal{P}^A&:=\Big\{\per(A,S)\mid S\in \{0,1\}^\alpha\Big\}
\end{align}
Thus $\mathcal W^A$ is the set of all $(A,S)$ games, i.e.\ all sequences satisfying recurrence relation \ref{eq:recurrence}, and $\mathcal P^A$ is the set of their periods. For general $A$, a natural open question is to find $\max(\mathcal P^A)$.

\subsection{Properties of Subtraction Games.}\label{sec:properties}

If the elements of $A$ are not coprime, we can interpret the sequence $w$ as multiple games $(w)_i$ proceeding in parallel. Formally, choose finite $A\subset \N_+$ with maximum $\alpha$ and $\gcd$ $1$, and denote $kA=\{kx\mid x\in A\}$. Then given some seed $S$ with $|S|=k\alpha$, where $S=\ang{S(m)\mid m\in\{0,\ldots,k\alpha-1\}}$, we break $S$ into classes modulo $k$ by defining $S_i$ for $i\in \{0,\ldots k-1\}$ such that for all $m\in \{0,\ldots \alpha-1\}$, we have $S_i(m):=S(mk+i)$. Thus $|S_i|=\alpha$ and $S$ can be decomposed into $S_i$'s. By analogy, similarly define $(w^{kA,S})_i(m):=w^{kA,S}(mk+i)$. We observe that $\{(w^{kA,S})_i(m)\}_{m=0}^\infty$ depends only on $S_i$.

\begin{proposition}[Multiplicative Linearity]\label{prop:linearity}
For any $k\in \N$, set $A$, and seed $S$ with $|S|=k\alpha$, then
\begin{equation}\label{eqthm:linearity}
(w^{kA,S})_i(m)=w^{A,S_i}(m)
\end{equation}
So if $S_i=S_0$ for all $i\in \{0,\ldots,k-1\}$, then $\per(kA,S)= k\per(A,S_0)$ and $\preper(kA,S)= k\preper(A,S_0)$. This condition holds for $S=1^{k\alpha}$, implying that $\per(kA)=k\per(A)$ and $w^{kA}(n)=w^A(\floor{n/k})$.
\hfill $\square$
\end{proposition}
Proposition \ref{prop:linearity} follows by applying the recurrence relation to $(w^{kA,S})_i(m)$.

\begin{definition}[Extension]
Choose set $A\subset \N_+$ and seed $S$. We say $b\in \N\setminus A$ is an \emph{extension} of $(A,S)$ if $w^{A\cup\{b\},S}= w^{A,S}$.
\end{definition}

\begin{proposition}[Better Definition of Extension]
For $b\in \N\setminus A$ is an extension of $(A,S)$ if and only if for all $n\in \N$ it holds that $w^{A,S}(n)=0\implies w^{A,S}(n-b)=1$.
\end{proposition}
\begin{proof}

Let $B=A\cup \{b\}$. Suppose it holds that $w^{A,S}(n)=0\implies w^{A,S}(n-b)=1$, but $w^{B,S}\neq w^{A,S}$, and choose the least $n\in \N$ where they differ. If $w^{A,S}(n)=1$, then for some $x\in A$, $w^{A,S}(n-x)=w^{B,S}(n-x)=0$, so $w^{B,S}(n)=1$. If instead $w^{A,S}(n)=0$, then $w^{A,S}(n-b)=w^{B,S}(n-b)=1$, so indeed $w^{B,S}=0$, contradicting that $w^{B,S}\neq w^{A,S}$.

In the other direction, suppose $w^{B,S}= w^{A,S}$ but there is some $n$ with $w^{B,S}(n) =w^{A,S}(n)=0$ and $w^{B,S}(n-b)=w^{A,S}(n-b)=0$. This contradicts the recurrence relation for $B$.
\end{proof}

This means that we can identify redundant elements of a set $A$ if they are extensions of the other elements. The following proposition is a digression, but 
\begin{proposition}
For all finite or infinite sets $A\subseteq \N_+$, if $\preper(A)=0$ and $\per(A)=p$, then $w^A=w^{A\cap \{1,\ldots,p\}}$, so every element of $A$ greater than $p$ is redundant as an extension of $A\cap \{1,\ldots,p\}$. There is no analogous statement for sequences with preperiods; for example $\{w^{\{1,4,7,\ldots\}}\}=0(101)^\infty$, but no other set generates that sequence. 
\end{proposition}
\begin{proof}For the first statement, let $B=A\cap \{1,\ldots, p\}$. Suppose for the sake of contradiction there is some least $n$ such that $w^A(n)\neq w^B(n)$. By the recurrence relation it must be that $w^A(n)=1$ and $w^B(n)=0$, and $w^A(n-a)=0$ for some $a\in A\setminus B$. By our definitions $a>p$ so $n>p$. Then $w^B(n-p)=w^A(n-p)=1$ so there is some $x\in B$ such that $w^B(n-p-x)=0$. Therefore $w^B(n-x)=0$, so $w^B(n)=1$, a contradiction.
For the second statement, suppose $w^A=0(101)^\infty$. Suppose $x\equiv 0\pmod 3$. Then $w^A(2)=w^A(2+x)=0$, so $x\notin A$. Suppose $x\equiv 2\pmod 3$. Then $w^A(0)=w^A(x)=0$, so $x\notin A$. Therefore $A\subseteq 3\N+1=\{1,4,7,\ldots\}$. It is easy to see that $A=3\N+1$ is a valid choice. Suppose $A\subsetneq 3\N+1$, and $3k+1$ is the least element of $3\N+1\setminus A$. Then surely $w^A(n)=w^{3\N+1}(n)$ for all $n< 3k+1$. However, for all $3j+1\in A$, $j<k$, $w^A(3k+1-(3j+1))=w^A(3(k-j))=1$, so $w^A(3k+1)=0$, a contradiction. Thus $3\N+1$ is the only set generating this sequence.\end{proof}
The following proposition gives another way to identify these extensions, which we use later in the paper.    
\begin{proposition}\label{prop:invariant-extension}
If $\{w^A\}$ is periodic over $p$ and $\preper(A)=0$, then for all $x\in A$ and $k\geq 1$, $b=kp+x$ is an extension of $A$.
\end{proposition}
\begin{proof}
Suppose $\per(A)\mid p$ and $\preper(A)=0$. Choose any $x\in A$ and $k\geq 1$, and $n\in \N$ such that $w^A(n)=0$. By the recurrence $w^A(n-x)=1$. If $n-x-kp\geq 0$, then by periodicity $w^A(n-x-kp)=1$. Otherwise, $w^A(n-x-kp)=1$ because we have no seed. Thus by the better definition $b$ is an extension of $A$.
\end{proof}

The following elementary example can be found in \cite[thm 1]{3elementsets}.
\begin{example}\label{ex:1}
If $A=\{1\}$, then $\{w^A\}=(01)^\infty$, so $\per(A)=2$ and $\preper(A)=0$. Because $1\in A$, any odd number $2k+1$ is an extension of $A$. 

From Example \ref{ex:1tok}, if $\{1,\ldots,k-1\}\subseteq B$ and $B\cap k\N=\emptyset$, then $\{w^B\}=(01^{k-1})^\infty$ and $\per(B)=k$. The set $B=\{1\}\cup\{p\mid p\text{ prime}\}$ is an example for $k=4$.
\end{example}

We must be careful to note that Proposition \ref{prop:invariant-extension} does not hold for all seeds. If we do allow for initial seeds $|S|\leq \alpha$, the induction step fails for $n\in [p+x-\alpha,p)$. This margin leads to many counterexamples; we provide two. Let $A=\{3,5\}$ and $S=0110$. We find $\per(A,S)=8$ and $\preper(A,S)=0$, so let $b=8+3=11$. Then:
\[\{w^{A,S}\}=(01^50^2)^\infty\hspace{20pt}\text{and}\hspace{20pt}\{w^{A\cup\{b\},S}\}=01^50101^3(01)^\infty\]
These are vastly different and we even caused a preperiod of length $12$. As another example:
\[\{w^{\{3,6,8\},0^21^3}\}=(011^3(01)^3)^\infty\hspace{20pt}\text{and}\hspace{20pt}\{w^{\{3,6,8,12\},0^21^3}\}=(01^401^3001^3)^\infty.\]

We also note that the converse of Proposition \ref{prop:invariant-extension} does not hold, though it appears to hold if $|A|\leq 3$. For example, let $A=\{1,2,6,11\}$. Then $\{w^A\} =0 \big(1^201^3(01^2)^2\big)^\infty$, so $\per(A)=12$, $\preper(A)=1$. However it is true that for all $k\geq 1$ and $x\in A$, $b=12k+x$ is indeed an extension of $A$. Also see the counterexample $\{1, 8, 13, 16\}$.


Contrasting these complex examples with Example \ref{ex:1}, we see that sequences are easiest to examine with no seed and no preperiod. The following Lemma can simplify the process of checking whether a sequence is in this form.

\begin{lemma}[Translating zeros]\label{lem:translating_zeros}
Given a set $A$, then for any $p\in \N$,  $w^A$ is periodic over $p$ and has no preperiod if and only if
it holds for all $x\in A$ and $m<x$ that 
\[w^A(m)=0\hspace{10pt}\implies\hspace{10pt} w^A(m+p-x)=1.\]
\end{lemma}
We call this ``translating zeros'' because 
$w(m)=0\implies w(m+p-x)=1$ suggests that  $w(m+p)=0$, but does not discuss the translation of the $1$'s. Note that 
it does not \emph{a priori} imply that the zeros translate, since only $m<x$ is considered.
\begin{proof} Suppose for the sake of contradiction that the premise $w(m)=0\implies w(m+p-x)=1$ holds but there is some least $m\in \N$ such that $w(m)\neq w(m+p)$. There are two cases.
\begin{enumerate}
    \item[(i)] If $w(m)=1$, then $w(m+p)=0$, so there is some $x\in A$ such that $w(m-x)=0$ but $w(m+p-x)=1$. Because there is no seed this implies $m-x\geq 0$, and because $w(m-x)\neq w(m-x+p)$ this contradicts the minimality of $m$. 
    \item[(ii)] If $w(m)=0$, then $w(m+p)=1$, so there is some $x\in A$ with $w(m+p-x)=0$ but $w(m-x)=1$. If $m-x\geq 0$, then this would contradict the minimiality of $n$. Otherwise, $m<x$ and $w(m)=0$, so the premise implies that $w(m+p-x)=1$, a contradiction.
\end{enumerate}
Either case leads to a contradiction, so for all $m\geq 0$, we have $w(m)=w(m+p)$. In the other direction, if $w$ is periodic with no preperiod then the condition follows by definition.
\end{proof}

\begin{corollary}\label{cor:translating_zerosabc}
If $A=\{a,b,c\}$ for any $a< b,c$, then $\per(A)\mid (b+c)$ and $\preper(A)=0$ if and only if $w^A(b+c-i)=1$ for all $i\in[1,a]$.
\end{corollary}
\begin{proof} Apply Lemma \ref{lem:translating_zeros} for $p=b+c$. For all $x\in A$, if $x=b$, then $w(m)=0\implies w(m+p-b)=w(m+c)=1$ by the recurrence. If $x=c$, then $w(m)=0\implies w(m+p-c)=w(m+b)=0$. If $x=a=\min(A)$, then we note that for all $m<a$, $w(m)=0$. Therefore it suffices to check that for $m\in [0,a-1]$, we have $w^A(b+c+m-a)=0$. Simplify by substituting $i=a-m\in [1,a]$. \end{proof}
This Corollary will likely help in proving Conjecture \ref{conj:abcperbc}.
\begin{corollary}\label{cor:translating_zeros1bc}
If $A=\{1,b,c\}$, then $\per(A)\mid (b+c)$ and $\preper(A)=0$ if and only if $w^A(b+c-1)=1$.\hfill $\square$
\end{corollary}
These corollaries give some insight to the usefulness of the translating zeros Lemma. For Corollary \ref{cor:translating_zeros1bc} we can determine the period and preperiod by checking one value of the sequence! This will be used in Section \ref{sec:1,b,c}.

\begin{lemma}\label{lem:pernopreper}
The following statements are equivalent.
\begin{enumerate}[(i)]
\item $w$ is periodic over $p$ with no preperiod
\item For all $n\geq 0$, $w(n)=w(n+p)$
\item For all $n\geq \alpha$, $\vec v(n)=\vec v(n+p)$
\item $\vec v(\alpha)=\vec v(\alpha+p)$
\item For all $x\in A$, for each $m<x$, we have $w(m)=0\implies w(m+p-x)=1$.\hfill $\square$
\end{enumerate}
\end{lemma}

\section{The $\{a,b\}$ case}\label{sec:ab}

We first inspect the case when $|A|=1$. If $A=\{a\}$, then the recurrence relation in Equation \ref{eq:recurrence} gives that $w^A(n)=1-w^A(n-a)$ for all $n\in \N$. This gives us an immediate characterization of the periodicity for all possible seeds.
\begin{proposition}
Let $A=\{a\}$. For all seeds $S$, let $\overline S$ be the string exchanging $0$'s and $1$'s. Then $\{w^{A,S}\} = \big(\,\overline S\,S\,\big)^\infty$, so $\per(A,S)\mid 2a$ and $\preper(A,S)=0$.\hfill $\square$
\end{proposition}

This gives the result that $\mathcal{W}^{\{a\}}=\big\{(\,\overline S\,S\,\big)^\infty\mid S\in \{0,1\}^a\}\big\}$ and thus $\left|\mathcal{W}^{\{a\}}\right|=2^a$. 
\begin{proposition}
Let $A=\{a\}$, and let $a=2^kc$, where $c$ is odd. Then $\mathcal P^A=\{2^{k+1}d: d\mid c\}$.\hfill$\square$
\end{proposition}
\begin{proof} First, we show that $p=\per(A,S)$ must be of the form $2^{k+1}d$. We know that $\{w^{A,S}\}$ will always be periodic over $2a$, so $p\mid 2a$. Additionally, we know $w^{A,S}(n)\neq w^{A,S}(n+a)$ for all $n$, so $p\mid\!\!\!\!/\,\, a$. Therefore $2^{k+1}\mid p$. We now show that for any $d\mid c$ we have $2^{k+1}d\in \mathcal P^{\{a\}}$. Because $c/d$ is odd let $c=(2x+1)d$. Then let $S=\left(1^{2^kd}0^{2^kd}\right)^x1^{2^kd}$, so we conclude $\{w^{A,S}\}=\left(\left(1^{2^kd}0^{2^kd}\right)^x1^{2^kd}\ \left(0^{2^kd}1^{2^kd}\right)^x0^{2^kd}\right)^\infty=\left(1^{2^kd}0^{2^kd}\right)^\infty$ and $\per(\{a\},S)=2^{k+1}d$. \end{proof}
These initial results are apparent at first inspection of the problem. With more work will get similarly general results for $|A|=2$. 
\begin{theorem}\label{thm:abperiod}
Let $A=\{a,b\}$.
For all seeds $S$, $\per(A,S)\mid a+b$ and $\preper(A,S)=0$.
\end{theorem}
\begin{proof}
Let $A=\{a,b\}$. For all $n\in \N$, in the case where $w(n)=0$ we have $w(n+a)=w(n+b)=1$, which implies that $w(n+a+b)=0$. Alternately, if $w(n)=1$, then $w(n-a)=0$ or $w(n-b)=0$, so by the $0$-case we know $w(n+b)=0$ or $w(n+a)=0$ respectively. This means $w(n+a+b)=1$.
\end{proof}

The following Theorem on $2$-set games with no seed is well known and can be found in \cite[p. 530]{winningways3}.

\begin{theorem}\label{thm:abstructure}
Let $A=\{a,b\}$ with $a<b$, and let $b=qa+r$ with $0\leq r<a$. Then
\begin{equation}\label{eqthm:abbasic}
\{w^A\}=\begin{cases}
(0^a1^a)^{q/2}0^r1^a &q\text{ is even.}\\
(0^a1^a)^{\ceil{q/2}}1^r &q\text{ is odd.}
\end{cases}\ =
\Big(\underbrace{0^a1^a0^a1^a\ldots}_{b}1^a\Big)^\infty
\end{equation}
so $\per(A)=2a$ if $b$ is an odd multiple of $a$ and $\per(A)=a+b$ otherwise.
\end{theorem}
\begin{proof}
For $n<b$, we note that by having no seed $w^{A}(n-b)=1$. This implies $w^{A}(n)=w^{\{a\}}(n)$. This yields the sequence $\{w^A\}=\underbrace{0^a1^a0^a1^a\ldots}_{b}$\ . For $n\in [b,a+b)$, we note that $n-b\in [0,a)$ so $w^A(n-b)=0$ and thus $w^A(n)=1$. Theorem \ref{thm:abperiod} implies that this period repeats.
\end{proof}

\subsection{A Preliminary Tool: Studying $\{1,b\}$}\label{sec:1,b}

Theorem \ref{thm:abstructure} characterizes the period structure for the default seed. Our next goal is to describe the sequence under all seeds. One corollary of the following theorems is that for any $a,b$ coprime, $4,6\notin \mathcal P^{\{a,b\}}$. At present, it seems like it should be possible to find a seed which generates a period of $4$ or $6$ as long as $4\mid a+b$ or $6\mid a+b$, but in fact it is not. We will give a full characterization of all period structures and lengths which will make this fact obvious. To do this, we first analyze the special case where $a=1$.

\begin{theorem}\label{thm:1bstructure}
Suppose we have $A=\{1,b\}$ and a string $X$ with no-subperiod. Then $X^\infty\in \mathcal W^A$ if and only if $|X|\Big|(1+b)$ and $X$ is some concatenation of $01$ and $011$ under some rotation.
\end{theorem}

\begin{proof}$\implies$ Suppose that $\{w^{A,S}\}=X^\infty$. By the assumption that $X$ has no sub-period, $\per(A)=|X|$, so by Theorem \ref{thm:abperiod} $|X|\Big|(1+b)$. Next we will show that $X$ has no two consecutive $0$'s or three consecutive $1$'s. By the recurrence relation, $w(n)=0$ implies $w(n+1)=1$. Additionally, suppose for some sufficiently large $n$ that $w(n)=w(n-1)=1$ This implies $w(n-b)=0$, so by Theorem \ref{thm:abperiod} $w(n-b+b+1)=w(n+1)=0$. These two restrictions mean that $X$ is a concatenation of $01$ and $011$ under some rotation.
\vskip 3pt 
$\Longleftarrow$ Suppose that $k|X|=(1+b)$ for $k\geq 1$ and $X$ is some concatenation of $01$ and $011$. For the seed $S$ we can simply choose $X^k$, so it suffices to show that $X^\infty$ satisfies the recurrence relation. Choose $n\in \N$. First, suppose $X^\infty(n-1)=0$. Because there are no consecutive $0$'s, $X^\infty(n)=1$, which satisfies the recurrence relation. Now suppose $X^\infty(n-1)=1$ and $X^\infty(n+1)=0$. Because there are no consecutive $0$'s, $X^\infty(n)=1$. Additionally because $X^\infty$ is periodic over $|X|$, $X^\infty(n+1-k|X|)=X^\infty(n-b)=0$, so indeed $X^\infty(n)=1$ satisfies the recurrence. Now consider the last case $X^\infty(n-1)=X^\infty(n+1)=1$. Because there are no three consecutive $1$'s, this implies $X^\infty(n)=0$. Additionally $X^\infty(n+1-k|X|)=X^\infty(n-b)=1$ and by assumption $X^\infty(n-1)=1$, so $X^\infty(n)=0$ satisfies the recurrence.
\end{proof}

Recall from Theorem \ref{thm:abperiod} that $\{w^{A,S}\}$ can never have a preperiod, so Theorem \ref{thm:1bstructure} completely characterizes all possible sequences for $\{1,b\}$. We also note that instead of considering strings $X$ with $|X|\big| (b+1)$ and no subperiod, we can equivalently consider all $X$ with $|X|=(b+1)$ and allow for any subperiod of length $p\mid (b+1)$. Define the sets
\[\mathcal Q:=\{X\in \{0,1\}^\ell\mid  \ell\in \N,\ X\text{ is a concatenation of 01 and 011 under some rotation}\},\] and $\mathcal Q(\ell):=\{X\in \mathcal Q\mid |X|=\ell\}$. Sequences of this form can be generated iteratively by hand or by computer. For example, $\mathcal Q(2)=\{01,10\}$, $\mathcal Q(3)=\{011,101,110\}$, and $\mathcal Q(6)=\{010101,101010,$\\$011011,101101,110110\}$.
Theorem \ref{thm:1bstructure} proves that $\mathcal W^{\{1,b\}}=\{X^\infty\mid X\in \mathcal Q(b+1)\}$.
We  can also provide an explicit enumeration of the possible sequences.

\begin{theorem}\label{thm:1bnumber}
Let $\phi$ be the plastic constant and $z,\bar z$ be the other two complex roots of $x^3-x-1$, and let $Q(\ell) :=\phi^\ell+z^\ell+{\bar z}^\ell$. Then the number of distinct sequences $\{w^{\{1,b\},S}\}$ over all seeds $S$ is
\begin{equation}
\big|\mathcal W^{\{1,b\}}\big|=\abs{\mathcal Q(1+b)}=Q(1+b).
\end{equation}
\end{theorem}
\begin{proof}
To find a recurrence relation, we consider the four possible forms for the period structure $X$ to take, and how each can be reduced to a different form. We will then add the sequences which count each of these forms.
\begin{align*}
Q_1(\ell)&&&X=0\ldots 01 &&Q_1(\ell-2)+Q_2(\ell-2)&&\text{by removing the last two numbers}\\
Q_2(\ell)&&&X=0\ldots 11 && Q_1(\ell-1)&&\text{by removing the last number}\\
Q_3(\ell)&&&X=1\ldots 01 && Q_2(\ell)&&\text{by rotating one character left}\\
Q_4(\ell)&&&X=1\ldots 10 && Q_1(\ell)+Q_2(\ell)&&\text{by rotating one character right}
\end{align*}
\[Q(\ell)=Q_1(\ell)+Q_2(\ell)+Q_3(\ell)+Q_4(\ell)=3Q_2(\ell)+2Q_1(\ell)=3Q_1(\ell-1)+2Q_1(\ell).\] We note $Q_1(\ell)=Q_1(\ell-2)+Q_1(\ell-3)$, so to follows that $Q(\ell)=Q(\ell-2)+Q(\ell-3)$. This is a linear recurrence relation, meaning $Q(\ell)$ is a linear combination of the form  $c_1\phi^\ell+c_2z^\ell+c_3{\bar z}^\ell$, where $\phi,$ $z$, and $\bar z$ are the roots of $x^3-x-1$. We give initial conditions $Q(1)=0$, $Q(2)=|\{01,10\}|=2$, and $Q(3)=|\{011,101,011\}|=3$.

To find the coefficients we compute that $Q(0)=3$, so $c_1+c_2+c_3=3$. Additionally, $Q$ is real so $c_2=c_3$. Finally, we know that $x^3-x-1=(x-\phi)(x-z)(x-\bar z)=x^3-(\phi+z+\bar z)x^2+\ldots$, so $\phi+z+\bar z=0=Q(1)$. Therefore $c_1=c_2=c_3=1$.
\end{proof}
$Q(\ell)$ is in the OEIS, known as the Perrin sequence \cite[A001608]{oeis}. 

\subsubsection{Distinct Periodicities.}\label{sec:distinctperiodicities}
If we analyze $\mathcal W^{\{1,5\}}$, or equivalently $\mathcal Q(6)$, we find that all of the sequences `look like' some form of $(01)^\infty$ or $(011)^\infty$, with some initial shift that we call a `rotation'. This is true for many sequences, motivating a formal notion of similarity between period structures.
\begin{definition}
We call two strings $X_1,X_2$ similar if they have a common subperiod, i.e.  $X_1^{k_1}=X_2^{k_2}$ for $k_1,k_2\geq 1$, or if $X_2$ is some rotation of $X_1$, i.e. $|X_1|=|X_2|=\ell$ and $X_2(n)=X_1(n+x\pmod\ell)$. We define $\simeq$ to be the equivalence relation generated by these two criteria. We say $X$ and $Y$ are distinct if $X\not\simeq Y$.
\end{definition}
\begin{definition}
We denote $\mathcal Q/\!\simeq$ as the set of equivalence classes of $\mathcal Q$ under $\simeq$. The length of a class $[Y]$ is the length of its smallest element, i.e. the smallest subperiod of $Y$ or the period length of $Y^\infty$.
\end{definition}
Recall that $\mathcal Q(\ell)$ is the set of all strings in $\mathcal Q$ with length $\ell$, which is in bijection with $\mathcal W^{\{1,\ell-1\}}$. Therefore $\mathcal Q(\ell)/\simeq$ is the set of distinct periodicites in $\mathcal Q(\ell)$, so
\begin{equation}\label{eqdef:w/sim}
\mathcal W^{\{1,b\}}/\!\simeq\ \ \ =\ \ \big\{X^\infty, [X]\in \mathcal Q(\ell)/\simeq\big\}
\end{equation}
\begin{definition}
We define $(\mathcal Q/\!\simeq)(\ell)$ to be set of distinct periodicities of $\mathcal Q/\!\simeq$ with length $\ell$. Similarly $(\mathcal W^A/\!\simeq)(\ell)$ is the set of distinct periodicites in $\mathcal W^A/\!\simeq$ with length $\ell$.
\end{definition}
It follows that
\begin{equation}\label{eqdef:w/sim(l)}
(\mathcal W^{\{1,b\}}/\!\simeq)(\ell)\ \ \ =\ \ \big\{X^\infty, [X]\in \mathcal (Q/\simeq)(\ell)\big\}.
\end{equation}
This implies that $p\in \mathcal P^{\{1,b\}}$ if and only if $p\mid (1+b)$ and $\abs{(\mathcal Q/\!\simeq)(p)}> 0$, as means there is some non-periodic string $X\in \mathcal Q$ with length $|X|=p$ and thus $X^\infty\in (\mathcal W^{\{1,b\}}/\!\simeq)(p)$. 

\begin{example}
If we want to find the distinct periodicities of length three, six, and eight, we observe that $(\mathcal Q/\!\simeq)(3)=\{[011]\}$, $(\mathcal Q/\!\simeq)(6)=\emptyset$, and $(\mathcal Q/\!\simeq)(8)=\{[01101101]\}$. In contrast, if we wish to know all distinct periodicites of $\{1,2\}$, $\{1,5\}$, and $\{1,7\}$, we instead consult $(\mathcal Q(3)/\!\simeq)=\{[011]\}$, $(\mathcal Q(6)/\!\simeq)=\{[(01)^3],[(011)^2]\}$, and $(\mathcal Q(8)/\!\simeq)=\{[(01)^4],[01101101]\}$.
\end{example}
It follows that because the periodicites of $\mathcal Q(\ell)$ must have some length $d\mid \ell$, we can enumerate $\mathcal Q(\ell)/\simeq$ by partitioning over period length. The following is a natural result of Equations \ref{eqdef:w/sim} and \ref{eqdef:w/sim(l)}.
\begin{equation}\label{eq:sumoverdivisors}
|\mathcal W^{\{1,\ell-1\}}/\simeq|=\sum_{d\mid \ell}|(\mathcal W^{\{1,\ell-1\}}/\simeq)(d)|\hspace{20pt}=\hspace{20pt}
|\mathcal Q(\ell)/\simeq|=\sum_{d\mid \ell}|(\mathcal Q/\simeq)(d)|
\end{equation}
In the rest of this section we will enumerate the sets $\mathcal W^{\{a,b\}}$, $\mathcal W^{\{a,b\}}/\!\simeq$, and $(\mathcal W^{\{a,b\}}/\!\simeq)(\ell)$ for all $\{a,b\}$. Recall that Theorem \ref{thm:1bnumber} gives $|\mathcal W^{\{1,b\}}|=Q(n)$.
\begin{proposition}\label{prop:distinctperiodicities}
Define
\begin{equation}\label{eq:distinctperiodicities}
N'(\ell):=\frac{\displaystyle Q(\ell)-\sum_{d!|\ell}d\cdot N'(d)}{\ell}\hspace{30pt}
\ncount (L):=\sum_{\ell\mid L}N'(\ell),
\end{equation}
where $d!|\ell$ represents proper divisors of $\ell$.
Then the number of distinct periodicities of $\{1,b\}$ of length $\ell\mid (1+b)$ is $|(\mathcal W^{\{1,b\}}/\!\simeq)(\ell)|=N'(\ell)$, and the total number of distinct perioditicites is $|\mathcal W^{\{1,b\}}/\!\simeq|=\ncount (b+1)$.
\end{proposition}
\begin{proof}
We will show $N'(\ell)=|(\mathcal Q/\simeq)(\ell)|$. First we consider all strings $X\in \mathcal Q(\ell)$. In particular suppose $X$ has period $p$, so $\ell/p=k$ and $X=Z^k$ for some $Z$ having no subperiod, i.e. $[Z]\in (\mathcal Q/\simeq)(p)$. Because $Z$ has no subperiod, it has $p$ different rotations and therefore $p$ representatives $X_1,\ldots, X_p\in \mathcal Q(\ell)$. Since $p$ can be any divisor of $\ell$, we conclude $|\mathcal Q(\ell)|=\sum_{p\mid \ell}p|(\mathcal Q/\simeq)(p)|$/ We can rearrange this to conclude $|(\mathcal W^{\{1,b\}}/\!\simeq)(\ell)|=|(\mathcal Q/\simeq)(\ell)|=N'(\ell)$ as written in Equation \ref{eq:distinctperiodicities}. Note that we do not need a base case to compute $N'$ explicitly because $1$ has no proper divisors. Equation \ref{eq:sumoverdivisors} implies $|\mathcal W^{\{1,b\}}/\!\simeq|=N(b+1)$.
\end{proof}
Note that $N'(\ell)$ does not depend on $b$, and in fact the set of distinct periodicities of length $\ell$ is equal for any $b$ as long as $\ell\mid (1+b)$. The OEIS contains sequences 
$N'(\ell)$ \cite[A113788]{oeis} and $\ncount (L)$ \cite[A127687]{oeis}, motivating an interesting bijection.

It is known that $Q(\ell)$ counts the maximal independent sets in vertex labeled cycles $C_{\ell}$ (see, for example, Example 1.2 in \cite{Furedi1987}).
In 2007, Bisdorff et. al demonstrated that $\ncount (b+1)$ counts the number of unlabeled maximal independent sets of the cycle $C_{b+1}$  \cite{independentsetsofcycle}. The correspondence between a binary sequence $Y$ and a vertex set $S\subseteq V(C_{b+1})$ is simple; we include the $i\nth$ vertex in $S$ exactly when $Y(i)=0$. For example $\mathcal Q(10)\ni (01)^2(011)^2\mapsto \{1,3,5,8\}\subseteq V(C_{10})$. The conditions for a string to be valid are equivalent to those for a maximal independent set. We cannot have two adjacent vertices of $S$ by independence, and $Y$ cannot have two consecutive zeros by Theorem \ref{thm:1bstructure}. We also cannot have three adjacent non-vertices in $S$ by maximality, or else we could add the middle vertex to $S$. Similarly $Y$ cannot have three consecutive ones. By our construction, it is clear that $\ncount $ counts these sets in $C_{b+1}$ up to rotation, but not reflection. For example, the sequence $(01)(011)(01)^2(011)(01)^3(011)\not \simeq (01)^3(011)(01)^2(011)(01)(011)$ because one cannot be rotated to the other, so they belong to different equivalence classes of $\mathcal Q/\simeq$. Thus, $\ncount(21)$ will count them both. However a reflection automorphism of $C_{21}$ would map the corresponding independent sets to each other. Moving forward, Table \ref{tab:enumeratingab} shows the sequences $Q(\ell)$, $N'(\ell)$, and $\ncount (\ell)$ for $\ell\in \{0,\ldots, 20\}$. 

\subsection{Generalizing to $\{a,b\}$}\label{sec:gentoab}

We will use a simple multiplicative permutation of $w^{A,S}$ to reduce every $\{a,b\}$ case to a version of $\{1,b'\}$. This will induce the strict structure of Theorem \ref{thm:1bstructure} on the seemingly complex periodicities of the $\{a,b\}$ case. We will generally assume that $a,b$ are coprime. If not, we can divide out $g=\gcd(a,b)$ and then find $g$ parallel copies of sequences $(w)_i$ for the coprime set $A=\{a/g,b/g\}$ using Multiplicative Linearity Proposition \ref{prop:linearity}.

\begin{definition}
Given some $a,b$ coprime, define the permutation $\sigma_{a,a+b}:\{0,1\}^{a+b}\biject\{0,1\}^{a+b}$ by $\sigma_{a,a+b}(Y)=X$, where $X(n)=Y(an\pmod {a+b})$ for all $n\in \{0,\ldots a+b-1\}$.
\end{definition}
Because $a$ and $a+b$ are coprime, $\sigma_{a,a+b}(Y)$ is a permutation of the string $Y$ for all $Y\in \{0,1\}^{a+b}$. This means $\sigma_{a,a+b}$ is invertable, so it is a permutation of the collection $\{0,1\}^{a+b}$.

\begin{theorem} [Bijection]\label{thm:2setreduction}
Suppose $A=\{a,b\}$ with $a,b$ coprime and let $A'=\{1,a+b-1\}$. For any string $Y$ with $|Y|=a+b$, let $X=\sigma_{a,a+b}(Y)$. Then $Y^\infty \in \mathcal W^A$ if and only if $X^\infty \in \mathcal W^{A'}$.
\end{theorem}
\begin{proof}
It suffices to show that $Y^\infty$ obeys the recurrence relation if and only if $X^\infty$ does, as we could then use the seeds $Y^{(a+b)/|Y|}$ and $X^{(a+b)/|X|}$ respectively. Let $a\inv $ be the multiplicative inverse of $a\pmod{a+b}$. This means $Y^\infty(n)=X^\infty(a\inv n\pmod{a+b})$, so $Y=\sigma_{a\inv, a+b}(X)$. By leveraging the periodicity of $Y^\infty$ and $X^\infty$ over $a+b$, and noting that $b\equiv -a\pmod{a+b}$, we get
\begin{align*}
Y^\infty(n)&&&=X^\infty(a\inv n)\\
Y^\infty(n-a)&=X^\infty(a\inv n-a\inv a)&&=X^\infty(a\inv n-1)\\
Y^\infty(n-b)&
=X^\infty(a\inv n-a\inv (-a))
=X^\infty(a\inv n+1)
&&=X^\infty(a\inv n-(a+b-1))
\end{align*}
Therefore $Y^\infty$ follows the recurrence relation of $A=\{a,b\}$ if and only if $X^\infty$ follows the recurrence relation of $A'=\{1,a+b-1\}$.
\end{proof}
Theorem \ref{thm:2setreduction} implies that the set of all period structures of $\{a,b\}$ can be obtained by applying the inverse permutation $\sigma_{a,a+b}\inv =\sigma_{a\inv ,a+b}$ to each period structure of $\{1,a+b-1\}$. Because $\sigma_{a,a+b}$ is a permutation of $\{0,1\}^{a+b}$, it bijects the set $\mathcal W^{\{a,b\}}$ of all $\{w^{\{a,b\},S}\}$ sequences with the set $\mathcal W^{\{1,a+b-1\}}$. In particular, we conclude precisely that if $a$ and $b$ are coprime, then
\begin{equation}
\mathcal W^{\{a,b\}}=\{Y^\infty \mid Y=\sigma_{a\inv ,a+b}(X),\ X\in \mathcal Q(a+b)\},
\end{equation}
or informally $\mathcal W^{\{a,b\}}=\sigma_{a\inv, a+b}\left[\mathcal W^{\{1,a+b-1\}}\right]$. Therefore Theorem \ref{thm:2setreduction} implies that we can generalize the results of Section \ref{sec:1,b}.
\begin{corollary}
Choose any $A=\{a, b\}=\{\tilde ag,\tilde bg\}$, where $g=\gcd(a,b)$. Then $|\mathcal W^A|=(Q(\tilde a+\tilde b))^g$.
\hfill $\square$
\end{corollary}

The power of $g$ follows from Linearity Proposition \ref{prop:linearity}. Each of the independent parallel sequences $(w^{A,S})_i(m)$ for $i\in \{0,\ldots,g-1\}$ is equal to some $\{\tilde a,\tilde b\}$ game. Thus we can count $g$-tuples of strings in $\sigma_{\tilde a\inv,\tilde a+\tilde b}\left[\mathcal Q(a+b)\right]$ to arrive at the total, which we formalize in Theorem \ref{thm:perioidicitieswithgcd}. Next, we find that $\sigma_{a,a+b}$ also bijects the classes of distinct periods.

\begin{corollary}\label{cor:permconservessubper}
Let $a,b$ be coprime and $A=\{a,b\}$. Then $|\mathcal W^{A}/\!\simeq(\ell)|=N'(\ell)$ is the number of distinct periodicities of length $\ell\mid(a+b)$, and $|\mathcal W^{A}/\!\simeq|=\ncount (a+b)$ is the total number of distinct periodicities of $A$.
\end{corollary}
\begin{proof}
It suffices to show that $\sigma_{a,a+b}$ preserves rotational symmetries and subperiods. Suppose some sequences $Y$ and $Y'$ are rotated copies of each other, i.e. there is some $r$ such that $Y(n)=Y'(n+r\pmod{\ell})$ for all $n$. This implies $(\sigma Y)(n)=(\sigma Y')(n+ar\pmod{\ell})$, so $\sigma Y$ and $\sigma Y'$ are rotated copies of each other with shift $ar$. Additionally, suppose $Y$ has some subperiod, i.e. there is some $r\not \equiv0\pmod{\ell}$ such that $Y(n)=Y(n+r\pmod{\ell})$. This implies $Y$ is a rotated copy of itself, so indeed $\sigma Y$ is a rotated copy of itself with shift $ar\not\equiv 0\pmod {\ell}$, since $\ell\mid (a+b)$ and $\gcd(a,a+b)=1$. We conclude that for all $X$ and $Y$, $X\simeq Y$ if and only if $\sigma X\simeq \sigma Y$, so equivalence classes of $\simeq$ are bijected by $\sigma_{a,a+b}$.
\end{proof}

If $a$ and $b$ are not coprime, it is more complicated to count the number of distinct periodicities, but we can use a generalization of Proposition \ref{prop:distinctperiodicities}.
\begin{theorem}\label{thm:perioidicitieswithgcd}
Choose any $A=\{a,b\}=\{\tilde ag,\tilde bg\}$, where $g=\gcd(a,b)$. Define the following functions for all $g,\ell\in \N\setminus \{0\}$ and $L=g\ell$:
\begin{equation}\label{eq:periodicitieswithgcd}
N'(L,g):=\frac{\displaystyle
Q(\ell)^g-\sum_{d!|L} d\cdot N'\left(d,\gcd(d,g)\right)}{L} \hspace{30pt}
\ncount (L,g):=\sum_{p\mid L}N'(p,\gcd(p,g))
\end{equation}
Then $|\mathcal W^A/\!\simeq(p)|=N'(p,\gcd(p,a,b))$ is the number of distinct periodicities of $A$ with length $p$, and $|\mathcal W^A/\!\simeq|=\ncount (a+b,g)$ is the total number of distinct periodicities of $A$.
\end{theorem}
\begin{proof}
First we will justify that the total number of strings in $\mathcal W^{\{a,b\}}$ which are periodic over $p$ is $\left(Q(\frac{p}{\gcd(p,g)})\right)^{\gcd(p,g)}$, not considering similarity under $\simeq$. Denote this set by $\mathcal W^{\{a,b\}}(p)$. Let $\ell=\tilde a+\tilde b$. We know by Theorem \ref{thm:abperiod} that $\{w\}$ is always periodic over $g\ell=a+b$, so for all $S$, $\{w^{A,S}\}=Y^\infty$ for some $Y$ with $|Y|=g\ell$. Using Linearity Proposition \ref{prop:linearity}, we can write $Y$ as a collection of parallel strings $Y_i$ for $i\in \{0,\ldots g-1\}$ where $|Y_i|=\ell$, and for all $m\in \{0,\ldots,\ell-1\}$, $Y_i(m)=Y(gm+i)$.
Additionally, for each $i$, $Y_i^\infty$ must satisfy the recurrence relation for $\{\tilde a,\tilde b\}$. Suppose that $Y$ is also periodic over length $p\mid \ell$, so $Y^\infty\in \mathcal W^{\{a,b\}}(p)$. Using the division algorithm let $p=xg+r$, where $r<g$ and $x\in \N$. 
Additionally, let $\gamma=\gcd(p,g)$ and $\tilde g=g/\gamma$. We will prove that the following conditions are sufficient and necessary for this to occur.
\begin{enumerate}[(i)]
\item For all $i\in \{0,\ldots,\gamma-1\}$, and all $m$, \ $\displaystyle{
Y_i(m)=Y_i(m+p/\gamma)}$
\item For all $i\in \{\gamma,\ldots, g-1\}$ and all $m$, \ $\displaystyle{Y_i(m)=\begin{cases}
Y_{i-r}(m-x)&i\geq r\\
Y_{i-r+g}(m-x-1)&i<r\\
\end{cases}}$
\end{enumerate}
We first show necessity. Assuming $Y$ is periodic over length $p$, this means for all $m, i$, we have 
\begin{align*}
Y_i(m)=Y(mg+i)&=Y(mg+i-p)=Y(mg+i-xg-r)=\begin{cases}
Y_{i-r}(m-x)&i\geq r\\
Y_{i-r+g}(m-x-1)&i<r\\
\end{cases}\\
&=Y(mg+i+\tilde g p)=Y(mg+i+g p/\gamma)=Y_i(m+p/\gamma),
\end{align*}
This proves stronger versions of (i) and (ii).

Next, we show sufficiency. Assume (i) and (ii) hold. For any $n$, use the division algorithm to get $n=mg+i$. We hope to show that $Y(n)=Y(n-p)$. Suppose that $i\in \{\gamma,\ldots,g-1\}$. Then by condition (ii),
\begin{equation}\label{eq:perswithgcdex}
Y(n)=Y_i(m)=\begin{cases}
Y_{i-r}(m-x)&i\geq r\\
Y_{i-r+g}(m-x-1)&i<r\\
\end{cases}=Y(mg-gx+i-r-gx)=Y(n-p)
\end{equation}
Alternately, suppose that $i\in\{0,\ldots,\gamma-1\}$, and consider the set of indices $i,i+r,i+2r,\ldots ,i+(\tilde g-1)r$, taken modulo $g$. We denote these $I_j\equiv i+rj\pmod g$ for $j\in \{0,\ldots, \tilde g-1\}$. Because $\gamma\mid r$, this means $I_j\equiv i\pmod \gamma$ for all $j$. Additionally, we know that the order of $r$ in the group $\Z_g^+$ is $g/\gcd(g,r)=\tilde g$, which means that $I_j\neq i$ for $j\in\{1,\ldots,\tilde g-1\}$. This means that only $I_0$ can be in the set $\{0,\ldots, \gamma-1\}$ and for all $j>0$ we must have $I_j\in \{\gamma,\ldots,g-1\}$. 

Now, let $n'=n-\tilde gp$, and consider the sequence of values $Y(n'),Y(n'+p),\ldots,Y(n'+(\tilde g-1)p)$. We know that $Y(n')=Y_i(m-p/\gamma)=Y_i(m)=Y(n)$ by condition (i). We also observe that for all $j\in \{0,\ldots,\tilde g-1\}$, we have $Y(n'+pj)=Y_{I_j}(m_j)$ for some $m_j$. Because we know $I_j\in \{\gamma,\ldots,g-1\}$, Equation \ref{eq:perswithgcdex} implies that $Y(n'+pj)=Y(n'+p(j-1))$. Therefore 
\[Y(n)=Y(n')=Y(n'+p)=\ldots =Y(n'+(\tilde g-1)p)=Y(n-p).\]
Therefore we have shown that conditions (i) and (ii) are equivalent to $Y^\infty$ being periodic over length $p$, and we now enumerate the strings satisfying these conditions. To satisfy (i) we may choose any $\gamma$-tuple of sequences in $\mathcal W^{\{\tilde a,\tilde b\}}$ which are periodic over $p/\gamma$ for $Y_0,\ldots,Y_{\gamma-1}$. To satisfy (ii) we must let $Y_\gamma,\ldots, Y_{g-1}$ be the appropriate translations of these sequences, so they are fixed. Thus we find the total number of possibilities $\abs{\mathcal W^{\{a,b\}}(p)}=(Q(p/\gamma))^\gamma$.

An argument similar to Proposition \ref{prop:distinctperiodicities} shows that $\abs{\mathcal W^{\{a,b\}}(p)}=\sum_{d\mid p}d\abs{\mathcal W^A/\!\simeq(d)}$. Letting \\$N'(d,\gcd(d,g))=\abs{\mathcal W^A/\!\simeq(d)}$, these can be rearranged to Equation \ref{eq:periodicitieswithgcd}. Note that we can also derive $N'(p,1)=N'(p)$ for all $p$, meaning that if $g=1$, we have reduced to the case of $\{a,b\}$ coprime.
\end{proof}

\begin{table}
    \centering
 \begin{tabular}{ |c|cccccccccccccccccccc| } 
 \hline
$\ell$   & 1 & 2 & 3 & 4 & 5 & 6 & 7 & 8 & 9 & 10 & 11 
& 12 & 13 & 14 & 15 & 16 & 17 & 18 & 19 & 20 \\
\hline
&&&&&&&&&&&&&&&&&&&&\\
$Q(\ell)$ & 0 & 2 & 3 & 2 & 5 & 5 & 7 & 10 & 12 & 17 & 22 & 29 & 39 & 51 & 68 & 90 & 119 & 158 & 209 & 277 \\
&&&&&&&&&&&&&&&&&&&&\\
$N'(\ell)$ & 0 & 1 & 1 & 0 & 1 & 0 & 1 & 1 & 1 & 1 & 2 & 2 & 3 & 3 & 4 & 5 & 7 & 8 & 11 & 13 \\
&&&&&&&&&&&&&&&&&&&&\\
$\ncount (\ell)$ &0&1&1&1&1&2&1&2&2&3&2&4&3&5&6&7&7&11&11&16\\
\hline
\end{tabular}
\caption{$Q(\ell)$ counts the possible sequences $\{w^{A,S}\}$ over all $S$ for $A=\{a,b\}$ if $a,b$ coprime and $a+b=\ell$. $N'(\ell)$ counts the number of distinct periodicities of length $\ell$. $\ncount (\ell)$ counts the number of distinct periodicities of any length for $A=\{a,b\}$.}
\label{tab:enumeratingab}
\end{table}
In the simpler case when $a,b$ are coprime, refer to the sequences $Q(\ell), N'(\ell),$ and $\ncount (\ell)$ in Table \ref{tab:enumeratingab}. As a result of the general enumeration we find that for all $g\mid L$, $N(L,g)=0$ exactly when $L=g$ or $L=4g$ or $(L,g)=(6,1)$. Thus
\[
\mathcal P^{\{a,b\}}=\Big\{p\ \Big|\  p|(a+b),\ p\not|\gcd(a,b),\ p\neq 4\gcd(p,a,b),\ (p,\gcd(p,a,b))\neq (6,1)\Big\}
\]
We also see that for any coprime $a,b$ and any $\ell\leq 10$, there is at most one distinct periodicity of $a$ and $b$ of length $\ell$ over all $S$. Given some $A=\{1,11k-1\}$ for $k\geq 1$, the two periodicites of length $11$ are $\mathcal (\mathcal Q/\!\simeq)(11)=\big\{\big[(01)^4(011)\big],\,\big[(01)(011)^3\big]\big\}$. This means that for any $\{a,b\}$, with $11\mid a+b$, the two periodicities are $\sigma_{a\inv ,11}\big((01)^4(011)\big)$ and $\sigma_{a\inv ,11}\big((01)(011)^3\big)$.

\begin{example}
Consider the set $A=\{3,11\}$. There are $\ncount (14)=5$ periodicites, where exactly $N'(14)=3$ have period $14$. First we write the 5 periodicities of $B=\{1,13\}$, given by \[\mathcal Q(14)/\simeq =\Big\{\left[(01)(011)^4\right],\ \left[(01)^4(011)^2\right],\ \left[(01)^3(011)(01)^1(011)\right],\ \left[((01)^2(011))^2\right],\ \left[(01)^7\right]\Big\}.\] Next, we permute these strings. $3\inv \equiv 5\pmod {14}$, so we compute $Y=\sigma_{5,14}(X)$. For the first string:
\begin{center}
\begin{tabular}{c|c|c|c|c|c|c|c|c|c|c|c|c|c|c|}
$n$&$\mathbf{0}$&$\mathbf{1}$&$\mathbf{2}$&$\mathbf{3}$&$\mathbf{4}$&$5$&$6$&$7$&$8$&$9$&$10$&$11$&$12$&$13$\\
\hline
$X$& $\mathbf{0}$&$\mathbf{1}$&$\mathbf{0}$&$\mathbf{1}$&$\mathbf{1}$&$0$&$1$&$1$&$0$&$1$&$1$&$0$&$1$&$1$\\
\hline
$\sigma(X)$&$\mathbf{0}$&$0$&$1$&$\mathbf{1}$&$1$&$0$&$\mathbf{0}$&$1$&$1$&$\mathbf{1}$&$0$&$1$&$\mathbf{1}$&$1$\\
\hline
$5n\pmod {14}$
&$\mathbf{0}$&$5$&$10$&$\mathbf{1}$&$6$&$11$&$\mathbf{2}$&$7$&$12$&$\mathbf{3}$&$8$&$13$&$\mathbf{4}$&$9$\\
\end{tabular}
\end{center}
This yields $\sigma_{5,14}\big(01(011)^4\big)=(0^21^3)^2(01^3)$. Repeating this procedure, we get the following:
\begin{align*}
\sigma_{5,14}\big(01(011)^4\big)&=&&=(0^21^3)^2(01^3)\\
\sigma_{5,14}\big((01)^4(011)^2\big)&= 01^30^31^3(01)^2&&\simeq 1^20^31^3(01)^3\\
\sigma_{5,14}\big((01)^1011(01)^3011\big)&=0^21^30^31^50&&\simeq 0^31^30^31^5\\
\sigma_{5,14}\big((01)^2011)^2\big)&=01^40^31^40^2&&\simeq (0^31^4)^2\\
\sigma_{5,14}\big((01)^7\big)&=&&=(01)^7
\end{align*}
\[
\mathcal W^{\{3,11\}}\!/\!\simeq\ \ =\Big\{\left[\big(\,(0^21^3)^2(01^3)\,\big)^\infty\right],\left[\big(\,1^20^31^3(01)^3\,\big)^\infty\right],\left[\big(\,0^31^30^31^5\,\big)^\infty\right],\left[(0^31^4)^\infty\right],\left[(01)^\infty\right]
\Big\}
\]
Note that the period lengths of $2$, $7$, and $14$ are conserved under the permutation, as shown in Corollary \ref{cor:permconservessubper}. This procedure yields all 5 periodicites of $\{3,11\}$.
\end{example}

The asympotic behavior of these functions is generally well behaved. Because $|z|=|\bar z|<1$, where $z$ and $\bar z$ are the non-real roots of $x^3-x-1$, we have the convergence $Q(\ell)=\phi^\ell+z^\ell+{\bar z}^\ell \underset{^{\ell\to\infty}}\to \phi^\ell$.\footnote{For all $\ell\geq 10$, we have the exact equality $Q(\ell)=\text{round}(\phi^\ell)$.} For the following analysis, let $A=\{a,b\}=g\{\tilde a,\tilde b\}$ with $g=\gcd(a,b)$, and assume that $\tilde a+\tilde b\geq 2.37\ln g$, which excludes relatively few sets. Then $\abs{\mathcal W^{\{a,b\}}}=Q(\tilde a+\tilde b)^g\sim \phi^{a+b}$. The set of seeds $\{S\mid S\in \{0,1\}^b\}$ has cardinality $2^b\geq \sqrt{2}^{a+b}$. Because $\phi\approx 1.32<\sqrt 2$ this means that the set of seeds grows faster than the sequences they generate, the number of seeds converging to the same sequence $\{w\}$ grows exponentially as $a+b$ increases. This approximation of $Q(\ell)$ also implies that
$N'(\ell)\sim \phi^\ell/\ell$ and $\ncount(\ell)\sim N(\ell)$. We can can also derive $\phi^L/L\sim N'(L,g)\sim \ncount (L,g)$, giving a strong estimation of the of the number of possible period structures for large $a,b$.

\section{The $\{1,b,c\}$ case}\label{sec:1,b,c}

Analogously to Section \ref{sec:1,b}, a good starting point for understanding the $\{a,b,c\}$ game is the $\{1,b,c\}$ game. This case was studied by Ho in \cite[sec.\ 2]{3elementsets}, where he solves the game for $c<4b$. We provide a full analysis by constructing $\{w^{\{1,b,c\}}\}$ for all $b$ and $c$ with no seed, where most importantly we specify the existence and structure of preperiods.

For the remainder of the paper we will use $x:\equiv y\pmod z$ to define $x$ as the least non negative remainder of $y$ modulo $z$.
\begin{theorem}\label{thm:1bc}
Suppose we have some $1<b<c$, and let $A=\{1,b,c\}$. Denote $q:=\floor{c/(b+1)}$, $r:\equiv c\pmod{b+1}$. This means $c=q(b+1)+r$. Additionally let $k:=b/2$ and $\gamma:=\frac{b-r-2}{2}$; we will only use these when they take on integer values.
\begin{center}
\begin{tabular}{|l|l|l|l|l|}
\hline
\!\!Case&Conditions&$\per(A)$&$\preper(A)$&$\{w^A\}$\\
\hline
i&$b,c$ odd,& $2$&$0$&$\left[01\right]^\infty$\\
\hline
ii&$b$ odd, $c$ even,& $b+c$&$0$&$\left[(01)^{c/2}1^{b}\right]^\infty$\\
\hline
iii&$b$ even, $c=b+1$&$2b$&$0$&$\left[(01)^k 1^b\right]^\infty$\\
\hline
iv&$b$ even, $r\in \{1,b\}$&$b+1$&$0$&$\left[(01)^k 1\right]^\infty$\\
\hline
v&$b$ even, $r>1$ odd &$b+c$&$0$&$\left[((01)^k1)^{q+1}1^{r-1}\right]^\infty$\\
\hline
vi&$b$ even, $r=b-2$&$c+1$&$0$&$\left[ ((01)^k1)^q(01)^{k-1}1\right]^\infty$\!\!\\
\hline
vii&$b$ even, $c>b+1$, $r<b-2$ even, $q> \gamma$&$c+1$&$\gamma(b+c+2)-b-1$&Equation \ref{eq:1bcpreper}\\
\hline
viii&$b$ even, $c>b+1$, $r<b-2$ even, $q< \gamma$&$b+c$&$q(b+c+2)-b-1$&Equation \ref{eq:1bcpreper}\\
\hline
ix&$b$ even, $c>b+1$, $r<b-2$ even, $q= \gamma$&$b-1$&$q(b+c+2)-a-1$&Equation \ref{eq:1bcpreper}\\
\hline
\end{tabular}
\end{center}
\end{theorem}
\begin{figure}
\centering
\includegraphics[scale=0.4]{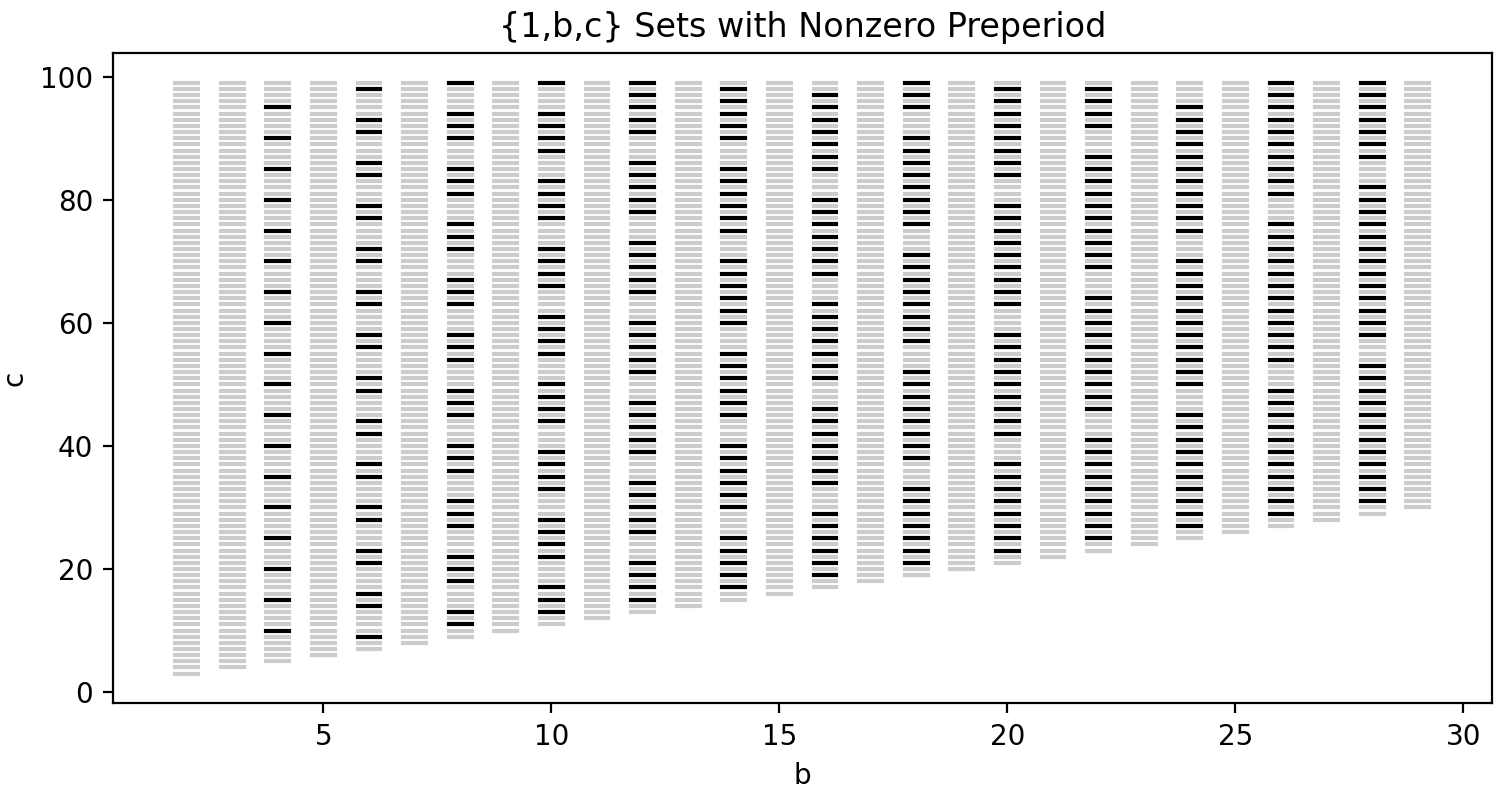}
\caption{Theorem \ref{thm:1bc} proves that $\preper(\{1,b,c\})\neq 0$ exactly when $b$ is even, $c>b+1$, and $r<b-2$ is even, with $r:\equiv c\pmod{b+1}$.}
\label{fig:1bcexistenceofpreperiod}
\end{figure}
The last three cases yield the most interesting results, providing an exact specification of the existence and structure of preperiods. These can be visualized in Figures \ref{fig:1bcexistenceofpreperiod} and \ref{fig:1bcpreperlength}.
\begin{proof}[Proof (i)]
If $b$ and $c$ are odd, then they are both extensions of $\{1\}$ as in Example \ref{ex:1}, so $\{w^A\}=\{w^{\{1\}}\}=(01)^\infty$.
\end{proof}
\begin{proof}[Proof (ii)]
Suppose $b$ is odd and $c$ is even.  For $n<c$, $w(n-c)=1$, so $w(n)=1-\min\{w(n-1),w(n-b),1\}$, and therefore $w^A(n)=w^{\{1,b\}}(n)=w^{\{1\}}(n)$, because $b$ is an extension of $\{1\}$. Thus $\{w^A\}=(01)^{c/2}\underset{c}{\ldots}$ Now for all $n\in [c,c+b)$, we find that $n-b<c$ and $n-c<c$. Therefore if $n$ is odd, then $n-b$ is even so $w(n-b)=0$ and therefore $w(n)=1$. If $n$ is even then $n-c$ is even so $w(n-c)=0$ and therefore $w(n)=1$. Thus:
\[
\{w^A\}=(01)^{c/2}1^{b}\underset{b+c}{\ldots}
\]
Because $w^A(b+c-1)=1$, Corollary \ref{cor:translating_zeros1bc} of the translating zeros lemma implies that this is the entire period with no preperiod.

\end{proof}
\begin{proof}[Proof (iii)]
Suppose $b=2k$ and $c=b+1$. For $n<c$ we have a single period of $\{1,b\}$, specifically $\{w^A\}=(01)^{k}1\underset{c}{\ldots}$ Next, for $n\in[b+1,2b)$, we find that $n-b<b+1$ and $n-c<b$. If $n$ is odd, then $n-c$ is even so $w(n-c)=0$ and thus $w(n)=1$. If $n$ is even, then $n-b$ is even so $w(n-b)=0$ and thus $w(n)=1$. This yields $\{w^A\}=(01)^k1^b\underset{b+c}{\ldots}$. Because $\vec v(b+c)=\vec v(0)=1^c$, this is the entire period with no preperiod, as implied by Lemma \ref{lem:pernopreper}.
\end{proof}
\begin{proof}[Proof (iv)]
Suppose $b=2k$ and $r\in \{1,b\}$. Then $c$ is an extension of $\{1,b\}$ by Proposition \ref{prop:invariant-extension}, so $\{w^A\}=\{w^{\{1,b\}}\}=\big((01)^k1\big)^\infty$.
\end{proof}
\begin{proof}[Proof (v)]
Suppose $b=2k$ and $r>1$ is odd. As above, for the first $c$ elements of the sequence, $\{w^A\}$ is equal to $\{w^\{1,b\}\}$.
\[
\{w^{\{1,2k\}}\}= \big((01)^k1\big)^q (01)^{\frac{r\text{-}1}{2}}0\underset{c}{\ldots}
\]
Now let $n=q(b+1)+m$ for $m\in [r,b)$. This means that $n-b<q(b+1)$ and $n-c=m-r$. If $m$ is odd, then $n-b$ is odd so $w(n-b)=0$, so $w(n)=1$. If $m$ is even, then $w(n-1)=1$ and $w(n-b)=1$ and $w(n-c)=w(m-r)=1$, so $w(n)=0$. This extends the sequence to
\[
\{w^A\}=\big((01)^k1\big)^q(01)^k\underset{q(b+1)+b}{\ldots}
\]
We just showed $w(q(b+1))=0$, so $w(q(b+1)+b)=1$.
\[
\{w^A\}=\big((01)^k1\big)^q(01)^k1\underset{(q+1)(b+1)}{\ldots}
\]
Now let $n=(q+1)(b+1)+m$ for $m\in [0,r-1)$. If $m$ is odd, then $n-b=q(b+1)+m+1$, which has an even remainder less than $b$, so $w(n-b)=0$ and therefore $w(n)=1$. If $m$ is even, then $n-c=b+1-(r-m)$, which is even and less than $b+1$, so $w(n-c)=0$ and therefore $w(n)=1$. We then extend the sequence.
\[
\{w^A\}=\big((01)^k1\big)^q(01)^k1^r\underset{b+c}{\ldots}
=\big((01)^k1\big)^{q+1}1^{r-1}\underset{b+c}{\ldots}
\]
Because $w^A(b+c-1)=1$, we may again apply Corollary \ref{cor:translating_zeros1bc} to claim that this is the complete period with no preperiod.

Note that if $r=1$, this solution still works and has sub-period $b+1$, which agrees with \textit{(iv)}.
\end{proof}
\begin{figure}
\centering
\includegraphics[scale=0.35]{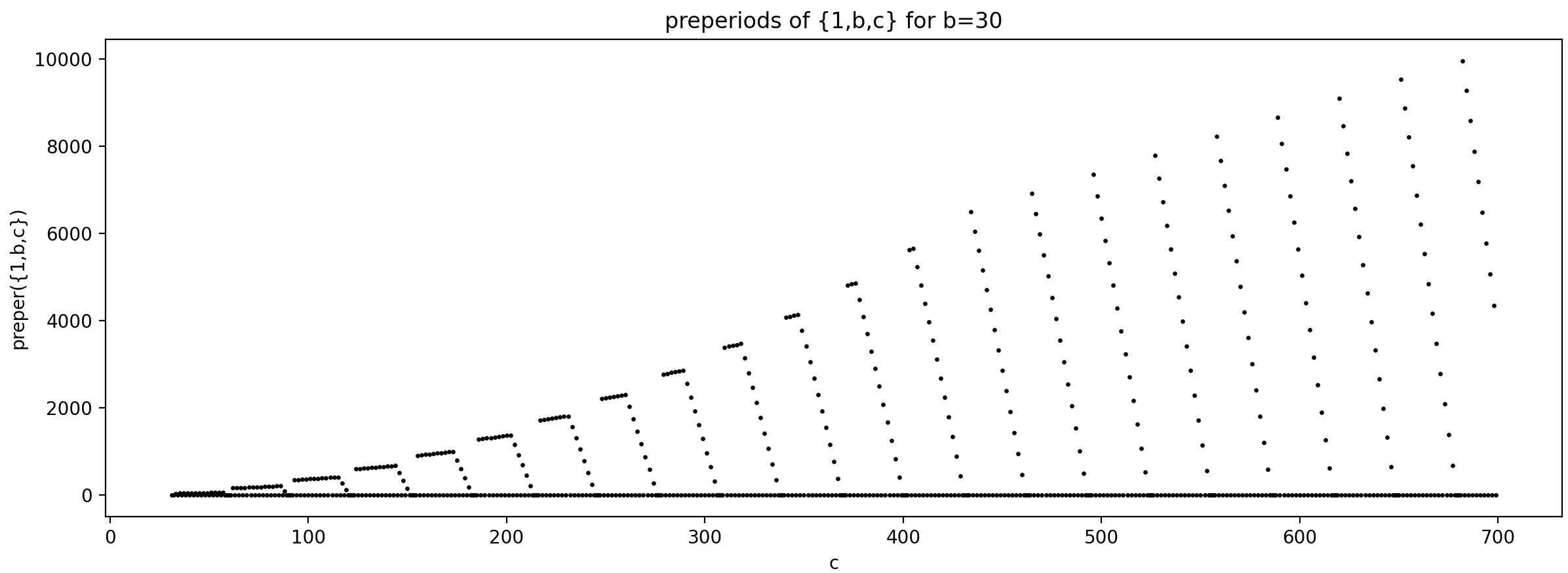}
\caption{Theorem \ref{thm:1bc} proves that $\preper(\{1,b,c\})=\min\left(\floor{\frac cb},\frac{a-r-2}{2}\right)(b+c+2)-b-1$ whenever it is nonzero. This is approximately quadradic in $c$ when $c$ is close to $b$, and transitions to linear when $c\gg b$.}
\label{fig:1bcpreperlength}
\end{figure}

\begin{proof}[Proof (vi, vii, viii, ix)]
Assume $b=2k$, $r$ is even, and $c>b+1$, and recall $\gamma=\frac{b-r-2}{2}$. All four of the remaining cases are expressed in the following equation. Interestingly, the preperiod structure is the same for all cases.
\begin{multline}\label{eq:1bcpreper}
\{w^A\}\!=\sum_{i=0}^{\min(q,\gamma)-1}
\Big(
((01)^k1)^{q-i}((01)^{k-1}1)^i(01)^{r/2+i}1^{2(\gamma-i)+1}(01)^{k-(\gamma-i)}1
\Big)\\
\left(\begin{cases}
((01)^k1)^{q-\gamma}((01)^{k-1}1)^{\gamma+1}& \gamma\leq q\\
((01)^{k-1}1)^q(01)^{r/2+q}1^{2(\gamma-q)+1}(01)^{r/2+q}1&\gamma\geq q
\end{cases}\right)^{\infty}
\end{multline}
A proof of this Equation is given in Subsection \ref{sec:1,b,cpreperproof}. We simply check that the recurrence relation is satisfied.

In case \textit{(vi)} where $r=b-2$, this means $\gamma=0$ so the summation is empty and there is no preperiod. Because we assume $q\geq 1$, this falls into the $\gamma\leq q$ case of Equation \ref{eq:1bcpreper}, which has length $q(b+1)+(b+1)-2=c+1$.

In the remaining cases both $q$ and $\gamma$ are nonzero, and each term in the preperiod summation has length
\begin{multline*}
(q-i)(b+1)+i(b-1)+(r+2i)+2(\gamma-i)+1+2(k-(\gamma-i))+1
\\=
q(b+1)-2i+r+2i+2k+2=c+b+2
.
\end{multline*}
Examine the last $b+1$ values of the last term in the summation, $2^{2(\gamma-\min(q,\gamma)+1)}(01)^{k-(\gamma-\min(q,\gamma)+1)}1$. If $\gamma\leq q$, this is $1^2(01)^{k-1}1$, which is equal to the last $b+1$ values of the period structure. If $q\leq \gamma$, this is equal to $2^{2(\gamma-q+1)}(01)^{k-\gamma+q-1}1=2^{2(\gamma-q+1)}(01)^{r/2+q}$, which is also equal to the last $b+1$ values of the period structure.
This implies that the preperiod transitions into the period $b+1$ steps earlier than depicted in Equation \ref{eq:1bcpreper}, and has length $\min(q,\gamma)(b+c+2)-b-1$.

In cases \textit{(vii)} and \textit{(viii)}, $\per(A)$ can be computed simply by counting the length of the strings in Equation \ref{eq:1bcpreper}. The period for $\gamma\leq q$ has length 
\[(b+1)(q-\gamma)+(b+1-2)(\gamma+1)=(b+1)q-2\gamma+b-2+1=(b+1)q+r+1=c+1.\]
The period for $q\leq \gamma$ has length 
\[q(a+1-2)+(r+2q)+(2\gamma-2q+1)+(r+2q)+1=q(a+1)+2r+2\gamma+2=b+a.\]
In case \textit{(ix)}, where $\gamma=q$, recall that $r+2\gamma=2k-2$. This allows us to simplify $r/2+q=r/2+\gamma=k-1$, so the two period structures are equivalent and can be expressed as $((01)^{k-1}1)^\infty$, which has period $b-1$. 

We might also notice that Equation \ref{eq:1bcpreper} also holds if $r=b$ (with no preperiod) and agrees with case \textit{(iv)}. In this case we would have $\gamma=-1$ so $\gamma\leq q$ and $\{w^A\}=((01)^k1)^\infty$.
\end{proof}
One result of this is that preperiod lengths can be arbitrarily longer than periods in the $\{1,b,c\}$ case. This can be seen in Figure \ref{fig:1bcpreperlength}, where the preperiod length appears quadratic in $c$.
\begin{example}
Let $b=2k$ and $c=k(b+1)$. Then $A=\{1,2k,2k^2+k\}$. This means $q=k$, $r=0$, and $\gamma=k-1$. Noting that $\gamma<q$, Equation \ref{eq:1bcpreper} yields
\begin{equation}\label{eq:ex1bc}
\{w^A\}\!
=\sum_{i=0}^{k-2}
\Big(
((01)^k1)^{k-i}((01)^{k-1}1)^i(01)^i1^{2(k-i)-1}(01)^{i+1}1
\Big)\\
\left(
(01)^k1((01)^{k-1}1)^{k}\right)^{\infty}
\end{equation}
and Theorem \ref{thm:1bc} (case vii) gives $\per(A)=2k^2+k+1$, and $\preper(A)=(k-1)(2k+2k^2+k+2)-2k-1=2k^3+k^2-3k-3$.

If we instead choose $b=2k$ and $c=(k-1)(b+1)$, we would have $q=\gamma$ (case ix) so $\per(A)=2k-1$ and $\preper(A)=2k^3-3k^2+2k-4$.
\end{example}
For general 3-sets, Alth\"ofer and B\"ulterman \cite[problem (vi)]{superlinear} provided the example $A=\{2s,4s+1,22s+2\}$ with $\per(A)=26s+3$, though they err in giving  $\preper(A)=24s^2-4s+1$ for $s\in [2,20]$ while actually $\preper(A)=24s^2-4s-1$ for all $s\in [2,200]$. Another example follows from \cite[thm 2]{bipartite}. If $A= \{k,k+2,2k+3\}$, then $\preper(A)=\frac{1}{2}(3k^2-5)$ and $\per(A)=2$. Thus for general $3$-sets, $\preper(A)$ is not bounded by any function of $\per(A)$, whereas for $A=\{1,b,c\}$, Theorem \ref{thm:1bc} shows that $\preper(A)=\bigO(\per(A)^3)$.

In Section $\ref{sec:ab}$, we used the $A=\{1,b\}$ case to characterize all possible $2$-sets using a permutation of $w^A$. A similar strategy may be possible if we could characterize all periodicities of $\{1,b,c\}$, though this would be more complicated. In particular, note that the length of the periods are highly dependent on seeds, unlike the $\{a,b\}$ case. An example of this is proven in Section \ref{sec:superpolynomial} and visualized in Figure \ref{fig:allperiods1bbp1}. 

Suppose $A=\{a,b,c\}$ and we are given any string $Y$ with $|Y|=p$. If $\gcd(a,p)=1$, then we let $a\inv $ be the multiplicative inverse of $a$ in $\Z^\times_p$ and $b':\equiv a\inv b\pmod p$ and $c':\equiv a\inv c\pmod p$. We see $X=\sigma_{a,p}(Y)$ is a permutation of $Y$, so we could show in a manner similar to Theorem \ref{thm:2setreduction} that $Y^\infty\in \mathcal W^{\{a,b,c\}}$ if and only if $X^\infty\in \mathcal W^{\{1,b',c'\}}$. Thus as long as $p$ and $a$ are coprime, we have $\mathcal W^{a,b,c}(p)=\sigma_{a\inv ,p}\left[\mathcal W^{1,b',c'}(p)\right]$. This observation carries little information without further understanding of $\mathcal W^{1,b,c}$.

As an example we apply Lemma \ref{lem:multiple_lines}. Choose any $n\in \N$ and $d\geq 1$. Let $b=4n+2d+1$ and let $p=2(n+1)b+1$. Lemma \ref{lem:multiple_lines} will imply that $p\in \mathcal P^{\{1,b,b+1\}}$. 2 is coprime to $p$, so calculate that $2\inv=(n+1)b+1$ and $2\inv b=(2n+d)+2\inv=(n+1)b+2n+d+1$, and finally $2\inv (b+1)=2n+d+1$. Therefore there is some seed $S$ such that $\per(\{2,2n+d+1,(n+1)b+2n+d+1\},S)=2(n+1)b+1$.

\subsection{Building the preperiod sequence for $\{1,b,c\}$}\label{sec:1,b,cpreperproof}

The statements in Theorems \ref{thm:1bc}, \ref{thm:aba+b}, and Lemma \ref{lem:multiple_lines} are somewhat tedious, so we provide three proofs of distinct flavors. In this section we provide a visual verification of Equation \ref{eq:1bcpreper} to complete the proof of Theorem \ref{thm:1bc}.

\begin{proof} We claim that if $b=2k$, $q=\floor{c/(b+1)}$, $r=c-qb$ is even $c>b+1$, and $\gamma=\frac{b-r-2}{2}$, then Equation \ref{eq:1bcpreper} holds.

To verify the construction of the preperiod, we will confirm that for each term $i\in\{0,\ldots, \min(q-1,\gamma)\}$, the recurrence relation holds. If $i=0$, note that the first $c$ entries proceed as $w^{1,b}$, so $\{w^A\}=\big((01)^k1)^q(01)^{r/2}\underset{^c}{\ldots}$, which agrees with the $i=0$ term of Equation \ref{eq:1bcpreper}. Thus when considering prefixes we may assume $i\geq 1$. Suppose the $i\nth$ term starts at entry $m$, and assume that the previous $i-1$ terms follow Equation \ref{eq:1bcpreper}.

The ``alignment diagram" on the left hand side in Figure~\ref{fig:alignment-diagrams}
re-writes the structure presented in Equation \ref{eq:1bcpreper} on two lines such that $w^A(n)$ on the first line is horizontally justified with $w^A(n-c)$ on the second line. To interpret this diagram, we simply confirm that for all $n$, if $w^A(n)=0$, then then look directly below to check that $w^A(n-c)=1$. Further, if $w^{A}(n)=1$, then either $w^{A}(n-1)=0$ on the left, or $w^A(n-c)=0$ directly below (shown in bold), or neither is true and we must have $w^A(n-b)=0$ (underlined). 

\begin{figure}
    \centering
\begin{sideways}
  \parbox{\textheight}{
\[
\hspace{-10pt}\{w^A(n)\}=\underset{m}{\ldots} \overbrace{0(10)^{k-1}1\textbf{1}\big((01)^{k-1}01\textbf{1}\big)^{q-i-1}}^{\big((01)^k1\big)^{q-i}}
\overbrace{\!\big((01)^{k-2}01\textbf{1}\big)^{i-1}(01)^{k-1}\hspace{58pt}\textbf{1}}^{\big((01)^{k-1}1)\big)^i}
\overbrace{(01)^{r/2+i-1}01(\textbf{1}\textbf{\underline{1}})^{\gamma-i}\textbf{1}}^{\big(01\big)^{r/2+i}1^{^{2(\gamma-i)+1}}}
\overbrace{(01)^{k-\gamma+i-1}01\textbf{1}}
^{(01)^{k-(\gamma-i)}1}\ldots
\]
\[
\hspace{-26pt}\{w^A(n-c)\}=\!\underset{m-c}{\ldots}\! 1(01)^{k-1}1
\underbrace{\!\textbf{0}\big((10)^{k-1}11
}_{\big((01)^{k}1\big)^{q-i-1}} \underbrace{\!\textbf{0}\big)^{q-i-1}\!\big((10)^{k-2}11}_{\big((01)^{k-1}1\big)^{i-1}}
\underbrace{\!\textbf{0}\big)^{i-1} (10)^{r/2+i-2}1^{2(\gamma-i)+4}}
_{\big(01\big)^{r/2+i-1}1^{^{2(\gamma-i+1)+1}}}
\underbrace{\textbf{0}\,(10)^{r/2+i-1}11}_{\big(01\big)^{k-(\gamma-i+1)}1}
\underbrace{(\textbf{0}\underline 1)^{\gamma-i}\textbf{0}(10)^{k-\gamma+i-1}11}
_{\big(01\big)^k1} \!\textbf{0}\ldots
\]

\vspace{4cm}

\[
\hspace{-10pt}\{w^A(n)\}=\underset{m}{\ldots}
\overbrace{(01)^{\gamma-i}(01)^{k-(\gamma-i)-1}011\big((01)^{k-1}011\big)^{q-i-1}}
^{\big((01)^{k}1\big)^{q-i}}
\overbrace{\big((01)^{k-1}1\big)\ 
\big((01)^{k-1}1\big)^{i-1}}
^{\big((01)^{k-1}1)\big)^i}
\,\overbrace{\,(01)^{r/2+i}(1\textbf{\underline 1})^{\gamma-i}\hspace{29pt}1}
^{\big(01\big)^{r/2+i}1^{^{2(\gamma-i)+1}}}
\,\overbrace{(01)^{k-(\gamma-i)-1}011}
^{(01)^{k-(\gamma-i)}1}\ldots
\]
\[
\hspace{-24pt}\{w^A(n-b)\}=\!\underset{m-b}{\ldots}\!
(11)^{\gamma-i}(1\underbrace{\!\!0)^{k-(\gamma-i)-1}11}
_{(01)^{k-(\gamma-(i-1))}1}
\!\underbrace{\!0\big((10)^{k-1}110\big)^{q-i-1}\hspace{7pt}(10)^{k-1}1\ 
\hspace{4pt}\big((1\!}
_{\big((01)^{k}1\big)^{q-i}}
\underbrace{\!\,0)^{k-1}1\big)^{i-1} \big((10)^{r/2+i}(1\textbf{0})^{k-r/2-1-i}
1\big)(1}
_{\big((01)^{k-1}1\big)^i}
\underbrace{\!0)^{r/2+i}\hspace{20pt}1}
_{\big(01\big)^{r/2+i}}\!
1*\ldots
\]

  }  
\end{sideways}
    \caption{Two alignments diagrams re-writing the structure presented in Equation \ref{eq:1bcpreper}. See text for details.}
    \label{fig:alignment-diagrams}
\end{figure}

Consider the additional case where $q\leq \gamma$ and $i=q$. The diagram nearly holds until the last $b+1$ entries. In particular, we note that $((01)^k1)^{q-i}=\emptyset$, so the $w^A(n-c)$ sequence should conclude with $(01)^{k-1}1(01)0$ instead of $(01)^{k}10$. Therefore the $q\nth$ term can be modified to \\$((01)^{k-1}1)^q(01)^{{r/2}+q}1^{2(\gamma-q)+1}(01)^{k-(\gamma-q)-1}10101$.

The alignment diagram on the right hand side in Figure~\ref{fig:alignment-diagrams}
similarly re-writes the structure such that $w^A(n)$ is horizontally justified with $w^A(n-b)$. We must confirm that if $w^A(n)=0$, then directly below, $w^A(n-b)=1$. We also confirm that the underlined entry has $w^A(n-b)=0$.

Therefore for all $i\in[0,\ldots, \min(q-1,\gamma)]$, the $i\nth$ term follows the recurrence. Note that the $*$ does not affect the recurrence, but represents that the value is unknown. This value is $1$ if $i<\gamma$ but $0$ if $i=\gamma$. In either case, the recurrence holds. Now we consider the $\gamma\leq q$ case. As justified above, after the $\gamma-1\nth $ term, the $\gamma\nth$ term begins at index $m_1$ with 
\begin{align*}
\{w^A\}&=\underset{m_1}{\ldots} \big((01)^k1\big)^{q-\gamma}\big((01)^{k-1}1)^\gamma(01)^{r/2+\gamma}1^{2(\gamma-\gamma)+1}(01)^{k-(\gamma-\gamma)}1\ldots \\
&=\underset{m_1}{\ldots} \big((01)^k1\big)^{q-\gamma}\big((01)^{k-1}1)^\gamma(01)^{k-1}1(01)^{k}1\ldots \\
&=\underset{m_1}{\ldots} \big((01)^k1\big)^{q-\gamma}\big((01)^{k-1}1)^{\gamma+1}(01)^{k}1\ldots 
\end{align*}
We now observe that for $Y=\big((01)^k1\big)^{q-\gamma}\big((01)^{k-1}1)^{\gamma+1}$, the sequence $Y^\infty$ satisfies the recurrence. Because $Y$ has length $c+1$, it suffices to check that $Y(n)=1$ if and only if $Y(n-1)=0$ or $Y(n+1)=0$ or $Y(n-2k)=0$. $Y$ satisfies this criterion by inspection.

Now consider the $q\leq \gamma$ case. As justified above, after the $q-1\nth$ term, the $q\nth$ term begins at index $m_1$ with
\begin{align*}
\{w^A\}&=\underset{m_1}{\ldots}
((01)^k1)^{q-q}((01)^{k-1}1)^q(01)^{r/2+q}1^{2(\gamma-q)+1}(01)^{k-(\gamma-q)-1}10101\ldots \\
&=\underset{m_1}{\ldots}
((01)^{k-1}1)^q(01)^{r/2+q}1^{2(\gamma-q)+1}(01)^{r/2+q}1\hspace{7pt} 0101\ldots 
\end{align*}
Now, we show that for $Y=((01)^{k-1}1)^q(01)^{r/2+q}1^{2(\gamma-q)+1}(01)^{r/2+q}1$, the sequence $Y^\infty$ satisfies the recurrence using alignment diagrams. Note that $|Y|=b+c$, and we start the diagram at $Y(-1)$ to simplify the diagram.
\begin{align*}
Y^\infty(n)&=\underset{-1}{\ldots}\hspace{7pt}1\overbrace{\!\big((01)^{k-2}01\textbf{1}\big)^{q-1}(01)^{k-1}\hspace{59pt}\textbf{1}}^{\big((01)^{k-1}1)\big)^q}
\overbrace{(01)^{r/2+q-1}01\hspace{1pt}(\textbf{1}\textbf{\underline 1})^{\gamma-q}\textbf{1}}^{\big(01\big)^{r/2+q}1^{^{2(\gamma-1)+1}}}
\overbrace{\!(01)^{r/2+q-1}\hspace{4pt}01\textbf{1}}
^{(01)^{r/2+q}1}\ldots
\\
Y^\infty(n-c)&=\!\!
\underset{-1-c}{\ldots} \underbrace{\textbf{0}\big((10)^{k-2}11}_{\big((01)^{k-2}1\big)^{q-1}}
\underbrace{\!\textbf{0}\big)^{i-1} (10)^{r/2+q-1}1^{2(\gamma-q)+2}}
_{\big(01\big)^{r/2+q}1^{^{2(\gamma-q)+1}}}
\underbrace{\textbf{0}\,(10)^{r/2+q-1}11}_{\big(01\big)^{r/2+q}1}
\underbrace{(\textbf{0}\underline 1)^{\gamma-q}\textbf{0}(10)^{k-\gamma+q-2}\hspace{1pt}11}
_{\big(01\big)^{k-1}1} \!\textbf{0}\ldots
\end{align*}
Similarly checking the recurrence for $b$:
\begin{align*}
Y^\infty(n)&=\underset{0}{\ldots}
\overbrace{\big((01)^{k-1}\hspace{53pt}1\big)
\big((01)^{k-1}1\big)^{q-1}}
^{\big((01)^{k-1}1)\big)^q}
\,\overbrace{\,(01)^{r/2+q}(1\textbf{\underline 1})^{\gamma-i}\hspace{29pt}1}
^{\big(01\big)^{r/2+i}1^{^{2(\gamma-i)+1}}}\,\overbrace{(01)^{r/2+q}1}
^{(01)^{r/2+q}1}\ldots
\\
Y^\infty(n-b)&=\underset{0-b}{\ldots}\hspace{3pt}
(11)^{\gamma-q}(1
\underbrace{\!0)^{k-\gamma+q-1}1\hspace{5pt}\big((1\!}
_{(01)^{r/2+q}1}
\underbrace{\!\,0)^{k-1}1\big)^{q-1} \big((10)^{r/2+q}(1\textbf{0})^{k-r/2-1-i}
1\big)(1}
_{\big((01)^{k-1}1\big)^q}
\underbrace{\!\!0)^{r/2+q}1}
_{\big(01\big)^{r/2+q}}\!
\ldots
\end{align*}
These diagrams verify that $Y^\infty$ satisfies the recurrence, and this completes the proof that Equation \ref{eq:1bcpreper} holds.
\end{proof}

\section{The $\{a,b,a+b\}$ case}\label{sec:a,b,a+b}

In \cite[p. 531]{winningways3}, Berlekamp et.\ al state without proof that the period lengths of the $\{a,b,a+b\}$ game are quadratic and give a formula for the period length of the Grundy sequence $\mathcal G(n)$. In \cite{superlinear}, Alth\"ofer and B\"ulterman prove a particular example of this case (Example \ref{ex:a2a+13a+1}). In this section we provide a proof for the general $\{a,b,a+b\}$ game by finding $\{w^A\}$ explicitly, which was presented as open by Ho in \cite[table 4]{3elementsets}.

\begin{theorem}\label{thm:aba+b}
Suppose $A=\{a,b,a+b\}$ with $a<b$. Define $k=\floor{\frac{b-1}{2a}}$, $\sigma_i:\equiv ib\pmod{a}$, and $\delta_i:
=\mathbf 1(\sigma_i>\sigma_{i-1})$, with the exception $\delta_1=0$.\footnote{This means $\delta_i=1$ if $\sigma_i>\sigma_{i-1}$ and $\delta_i=0$ if not.} Let $\tilde a=a/\gcd(a,b)$. Then
\begin{equation}
\{w^A\}=\begin{cases} \label{eqthm:aba+bstructure}
\displaystyle \left(\sum_{i=1}^{\tilde a-1}\left((0^a1^a)^{k-\delta_i}0^{\sigma_i}1^b0^{a-\sigma_i}1^a\right)0^a1^{a+b}\right)^\infty
&1\leq b<a\pmod{2a}
\\
\left((0^a1^a)^k0^a1^{a+b}\right)^\infty&\text{Otherwise}
\end{cases}
\end{equation}
It follows that
\begin{equation}\label{eqthm:aba+b}
\per(A)=\begin{cases}
\tilde a(2b-2ka)
&
1\leq b<a\pmod{2a}
\\
b+2a\left(k+1\right)
&
\text{Otherwise}
\end{cases}
\end{equation}
\end{theorem}
\begin{figure}
    \centering
    \includegraphics[scale=0.57]{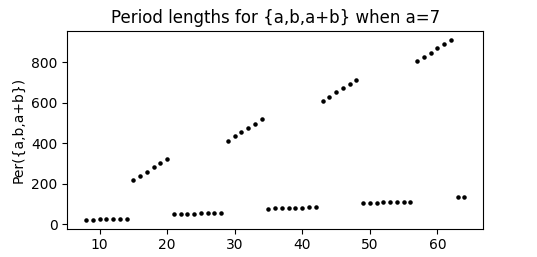}
    \!\!\!\!\!\!\!\!\!\!\!    \!\!\!\!\!
    \includegraphics[scale=0.55]{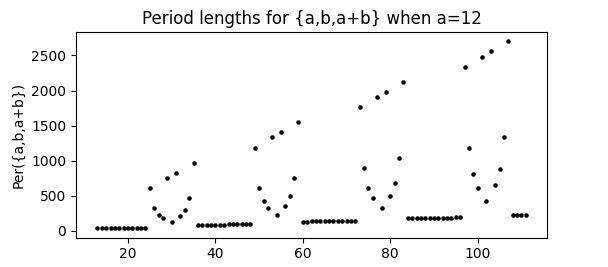}
    \caption{Two plots of $\per(\{a,b,a+b\})$ with fixed $a$. Note the phase transition from linear period lengths to quadratic, as well as the dips when $\gcd(a,b)\neq 1$. See Figure \ref{fig:abapb_extreme} for a more extreme example.}
    \label{fig:abapb}
\end{figure}
Berlekamp et.\ al \cite[p. 531]{winningways3} inspire an alternate formulation of Equation \ref{eqthm:aba+b}. If we let $b=2ha+\rho$ for $\rho\in(-a,a]$, then the following also holds:
\begin{equation*}
\per(A)=\begin{cases}
\tilde a(2b+\rho)&1\leq \rho<a
\\
\ \ \,2b+\rho&\rho\leq 0\text{ or }\rho=a
\end{cases}
\end{equation*}

\begin{proof}
We separate this proof into $2$ cases, beginning with the simpler linear case.

\textbf{Case 1:}
If $b\geq a\pmod{2a}$ or $b\equiv 0\pmod{2a}$, we can let $b=qa+r$ for odd $q=2k+1$ and $r\in [0,a]$.

For $n<b$ we find $w^A(n)=w^{\{a\}}(n)$, and $\{w^{\{a\}}\}=(0^a1^a)^\infty$. Because $q$ is odd,
\begin{align*}
\{w^A\}=\ 
&\ \left(0^a1^a\right)^k 0^{a}
1^r\ \underset{b}{\ldots}
\end{align*}
Next, for all $n\in[b,b+qa)$, we see that if $w(n-b)=0$ then both $w(n)=1$ and $w(n+a)=1$. This fills in the next portion of the sequence; we use underline and bold characters to highlight the structure of the recurrence.
\begin{align*}
\{w^A\}=\ 
&\ \left(\underline 0^a1^a\right)^k \underline 0^{a}
1^r
\left(\underline 1^{a}\textbf{1}^{a}\right)^k\underline 1^{a}\textbf{1}^{a}
\ \ =\ \ \left(0^a1^a\right)^k 0^{a}1^{a+b}\underset{b+qa+a}\ldots
\end{align*}
We now observe that $\vec v\big(b+(q+1)a\big)=1^{a+b}=\vec v(0)$, which proves that $\per(A)=b+2a(k+1)$ with no preperiod. We now move to the quadratic case. 

\vskip 5pt

\textbf{Case 2:} If $1\leq b<a\pmod{2a}$, then we will have $b=qa+r$, for $r\in [1,a)$ and $q=2k$. Again for $n<b$, the sequence is identical to $w^{\{a\}}$.
\begin{align*}
\{w^A\}=\ 
&\ \left(0^{a}1^{a}\right)^k0^r\underset{b}{\ldots}
\end{align*}
Once again for $n\in [b,2b)$, we notice that if $w^A(n-b)=0$, then by the recurrence $w^A(n)=w^A(n+a)=1$. This means that for all $n\in [b,2b)\cup[(2b-r)+a,2b+a)$, we have $w(n)=1$, as illustrated in the equation below. This leaves a gap of length $a-r$. In this gap $n\in [2b,(2b-r)+a)$, we find $w(n-a)=w(n-b)=w(n-a-b)=1$, so indeed $w(n)=0$.
\begin{align*}
\{w^A\}=\ 
&\ \undb{\left(\underline 0^{a}1^{a}\right)^k\underline 0^r}_b \undb{(\underline 1^a\textbf1^a)^k\underline 1^r}_b0^{a-r}{\ } \textbf1^r\underset{2b+a}{\ldots}
\ \ =\ \ 
\left(0^{a}1^{a}\right)^k0^r1^b0^{a-r}{\ }1^r\underset{2b+a}{\ldots}
\end{align*}
Because we showed $n\in [2b,2b+a-r)$ has $w(n)=0$, this implies $w(n+a)=1$, so we may append $(a-r)$ $1$'s after $2b+a$.
\begin{align*}
\{w^A\}=\ 
&\ 
\left(0^{a}1^{a}\right)^k0^r1^b0^{a-r}{\ }1^a\underset{2b+2a-r}{\ldots}
\end{align*}
This completes the first term in the summation, where $\sigma_1=r$ and $\delta_1=0$. Using this as a base case for induction, we now show that the $i+1\nth$ term in the summation succeeds the $i\nth$ term. Suppose
\begin{align*}
\{w^A\}=\ 
&\ 
\ldots \left(0^{a}1^{a}\right)^{k-\delta_i}0^{\sigma_i}1^b0^{a-\sigma_i}1^a\ldots
\end{align*}
We need only consider recent values of $w(n)$ back to $w(n-a-b)$, so define the index $m$ such that\\
$\{w^A\}=\underset{m}{\ldots}1^{b-a+\sigma_i}0^{a-\sigma_i}{\ }1^a\ldots$ Now let $m_1=m+(b-a+\sigma_i)$ and notice that $\vec v(m_1)$ ends with a long string of ones. 
In particular we find that for all $n\in[m+a+b,m_1+b)$, we have $w^A(n-a-b)=1$ and $w^A(n-b)=1$, so $w^A$ proceeds identically to $w^{\{a\}}$ for a length of precisely $b-2a+\sigma_i$. We must now consider sub-cases for parity, where $\sigma_i+r<a$ or $\sigma_i+r\geq a$. We show the two cases below
\begin{align*}
\{w^A\}=\ 
&\ 
\ldots \left(0^{a}1^{a}\right)^{k-\delta_i}0^{\sigma_i}1^b\undb{0^{a-\sigma_i}1^a
\ \ \ 
0^a1^a0^a1^a\ldots}_{b}\underset{m_1+b}{\ldots}
\\=\ &\ 
\begin{cases}
\underset{m_1}{\ldots}\,{0^{a-\sigma_i}1^a
\ \ \ 
(0^a1^a)^{k-1}0^{\sigma_i+r}}\hspace{20pt}\underset{m_1+b}{\ldots}&\sigma_i+r< a
\\
\underset{m_1}{\ldots}\,{0^{a-\sigma_i}1^a
\ \ \ 
(0^a1^a)^{k-1}0^a1^{\sigma_i+r-a}}\underset{m_1+b}{\ldots}&\sigma_i+r\geq a
\end{cases}
\end{align*}
Note that in the two cases we can plug in $\sigma_{i+1}=\sigma_i+r=$ and $\sigma_{i+1}=\sigma_i+r-a$ respectively.
In the second case, where $\sigma_i+r\geq a$, then for $n\in \big[(m_1+b),(m_1+b)+(a-\sigma_{i+1})\big)$, we find $w^A(n-a)=0$ so $w^A(n-a)=1$. For $n\in \big[(m_1+b)+(a-\sigma_{i+1}),(m_1+b)+a\big)$, we find $w^A(n-a)=w^A(n-b)=w^A(n-a-b)=1$, so $w(n)=0$. Extend this second case below, defining $m_2$ as the frontier of the sequence:
\begin{align*}
\{w^A\}=\ 
&\ 
\begin{cases}
\underset{m_1}{\ldots}\,{0^{a-\sigma_i}1^a
\ \ \ 
(0^a1^a)^{k-1}0^{\sigma_{i+1}}}\hspace{20pt} \underset{m_2}{\ldots}&\sigma_i+r< a
\\
\underset{m_1}{\ldots}\,0^{a-\sigma_i}1^a
\ \ \ 
(0^a1^a)^{k-1}0^a\mathbf1^a0^{\sigma_{i+1}} \underset{m_2}{\ldots}&\sigma_i+r\geq a
\end{cases}
\end{align*}
This adds another copy of $0^a1^a$ to the sequence in the second case, and observe that $\delta_{i+1}=1$ in the first case and $\delta_{i+1}=0$ in the second. We can therefore combine cases. Additionally, we find that for $n\in[m_2,m_2+b)$, if $w(n-b)=0$ then $w(n)=w(n+a)=1$.
This leads to another string of $1^b$ with a gap of length $a-\sigma_{i+1}$ which is filled with zeros. This is shown below, where $m_3=m_2+b+a$,
\[
\{w^A\}=\ldots\,(\underline 0^a1^a)^{k-\delta_{i+1}}\underline 0^{\sigma_{i+1}} \undb{(\underline 1^a\textbf 1^a)^k \underline 1^{\sigma_{i+1}}}_b \undb{0^{a-\sigma_{i+1}} \textbf 1^{\sigma_{i+1}}}_a \underset{m_3}{\ldots}
\]
For $n\in \big[m_3,m_3+(a-\sigma_{i+1})\big)$, we see $w(n-a)=0$ so $w(n)=1$. This completes the inductive step as we show below:
\[
\{w^A\}=\ldots
\left(0^{a}1^{a}\right)^{k-\delta_i}0^{\sigma_i}1^b0^{a-\sigma_i}1^a
\ \ \ 
\left(0^{a}1^{a}\right)^{k-\delta_{i+1}}0^{\sigma_{i+1}}1^b0^{a-\sigma_{i+1}}1^a\ldots 
\]
By induction this pattern will repeat. Let $\tilde a=\frac{a}{\gcd(a,b)}$, which is the order of $b$ in $\Z_a^+$. Thus the $\tilde a\nth$ iteration is the first such that $\sigma_{\tilde a}=0$ and $\delta_{\tilde a}=0$ so the sequence is 
\begin{equation}
\{w^A\}\ \ =\ \ \sum_{i=1}^{\tilde a-1}\Big((0^a1^a)^{k-\delta_i} 0^{\sigma_i}1^b0^{a-\sigma_i}1^a\Big)\ (0^a1^a)^k1^b\underset{m_4}{\ldots}
\end{equation}
We see that $\vec v(m_4)=1^{a+b}=\vec v(0)$, so this is the entire period. We see that the $i\nth$ term of the summation has length $2ka-2a\delta_i+b+2a$, and there are $\tilde a$ total terms. We also compute that $\frac{r\cdot \tilde a}{a}$ is the number of terms $j$ for which $\sigma_j<\sigma_{j-1}$, so if we include the first term $\delta_1=0$ and exclude the last term $\delta_{\tilde a}=0$, there are precisely $r/\gcd(a,b)$ terms in the summation where $\delta_j=0$, and the rest have $\delta_i=1$.
Therefore $\sum_{i=1}^{\tilde a-1}\delta_i=(\tilde a-1)-r/\gcd(a,b)$. We conclude that the total period length is
\begin{align*}
\per(A)&=\sum_{i=1}^{\tilde a-1}(2ka-2a\delta_i+b+2a)+2ka+b
\\&=(2ka(\tilde a-1)-2a((\tilde a-1)-r/\gcd(a,b))+b(\tilde a-1)+2a(\tilde a-1))+2ka+b
\\
&=2ka\tilde a+2ar/\gcd(a,b)+b\tilde a
\\
&=3\tilde ab-2ka\tilde a\qedhere
\end{align*}
\end{proof}
The following example is given as a Theorem in \cite{superlinear}.
\begin{figure}
\centering
\begin{verbatim}
-------------1111111111111---11111111111111111111111111111----------1111111111111
                          ------11111111111111111111111111111-------1111111111111
                          ---------11111111111111111111111111111----1111111111111
                          ------------11111111111111111111111111111-1111111111111
-------------1111111111111--11111111111111111111111111111-----------1111111111111
                          -----11111111111111111111111111111--------1111111111111
                          --------11111111111111111111111111111-----1111111111111
                          -----------11111111111111111111111111111--1111111111111
-------------1111111111111-11111111111111111111111111111------------1111111111111
                          ----11111111111111111111111111111---------1111111111111
                          -------11111111111111111111111111111------1111111111111
                          ----------11111111111111111111111111111---1111111111111
                          -------------111111111111111111111111111111111111111111
\end{verbatim}
    \caption{Let $a=13$, $b=2a+3$, and $A=\{13,29,42\}$. We compute $\per(A)=793$, and the period structure is shown above, with \texttt{-} used for $0$.}
    \label{fig:enter-label}
\end{figure}

\begin{example} \cite[Thm 3.1]{superlinear}\label{ex:a2a+13a+1} Let $A=\{a,2a+1,3a+1\}$. Then Equation \eqref{eqthm:aba+bstructure} gives the following sequence, where $b=2a+1$, $k=1$, and $\sigma_i=i$ so $\delta_i=1$ for all $i\in [1,a-1)$.
\begin{align*}
\{w^A\}\ \ &=\ \ 
\Bigg(0^a1^a\sum_{i=1}^{a-1}\Big( 0^{i}1^b0^{a-i}1^a\Big)\ 0^a1^{a+b}\Bigg)^\infty\\
&=\ \ 
\Bigg(0^a1^a\ \ 0^11^b0^{a-1}1^a\ \ \ 
0^{2}1^b0^{a-2}1^a\ \ \ 
0^{3}1^b0^{a-3}1^a\ \ \ \ldots\ \ \ 0^a1^{a+b}\Bigg)^\infty
\end{align*}
Additionally $\per(A)=(6a^2+3a)-2a^2=4a^2+3a\sim \frac{2}{9}\max(A)^2$.
\end{example}

This class of sets appears to be the only case for $|A|=3$ which can have superlinear period lengths with no seed. Further generalizations using this result and Theorem \ref{thm:1bc} might bring the following conjecture within reach.

\begin{conjecture}\label{conj:linearbound}
For all $a<b<c$ such that $a+b\neq c$,
\begin{equation}
\per(\{a,b,c\})<2c
\end{equation}
\end{conjecture}
We have verified Conjecture \ref{conj:linearbound} computationally for all $\{a,b,c\} \in{[200]\choose 3}$. If this conjecture is true, it must relate in part to some invariant of the structure caused by the initial seed of $1^\alpha$, because it fails entirely for different seeds, even if $a,b,$ and $c$ are coprime. This invariant would therefore be carried through the quadratic length preperiods seen in Section \ref{sec:1,b,c}.

\begin{figure}
    \centering
    \includegraphics[scale=0.35]{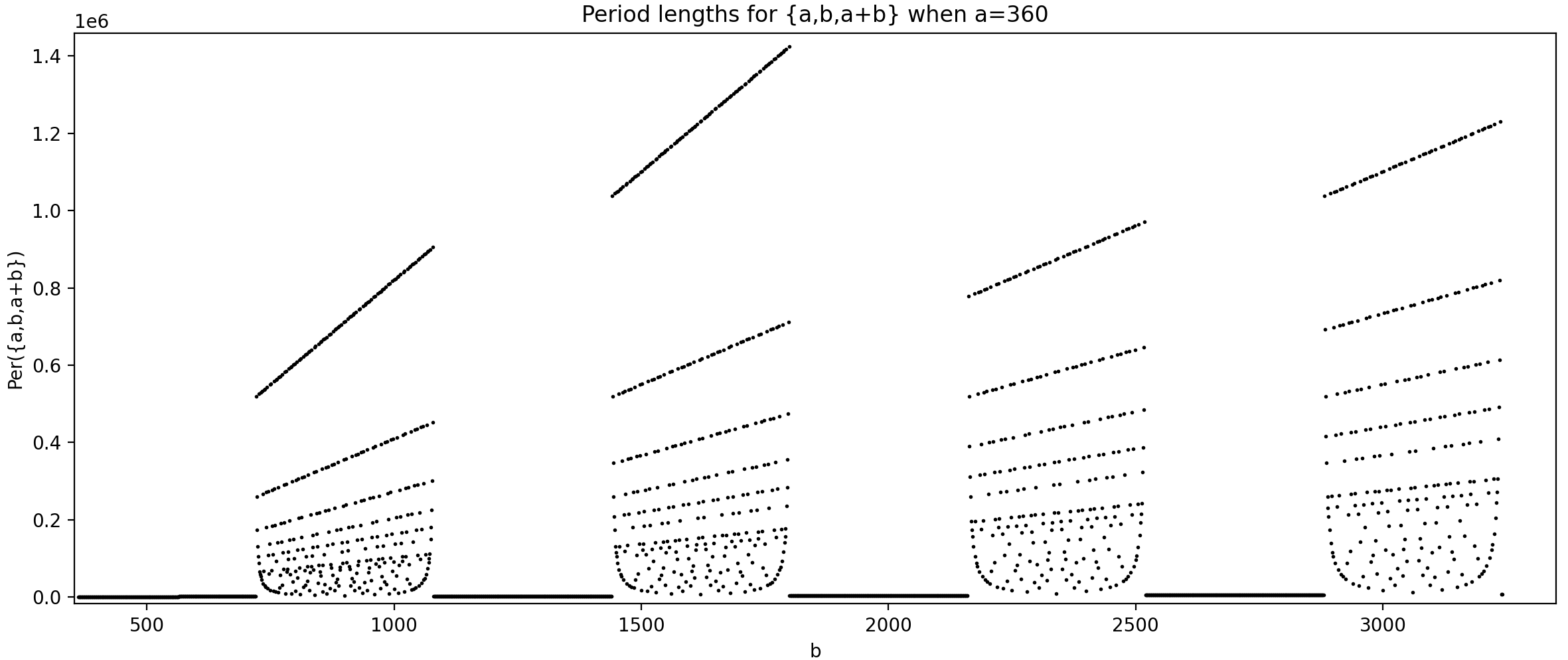}
    \caption{The period lengths for $\{a,b,a+b\}$ where $a=360$, a superior highly composite number.}
    \label{fig:abapb_extreme}
\end{figure}

\section{Super-polynomial Period lengths with initial Seeds.}\label{sec:superpolynomial}

In this section we consider a particular family of sets which demonstrate that super-polynomial period lengths exist for 3-sets, given properly chosen initial seeds. This family will relate to an intersection of the cases of Sections \ref{sec:1,b,c} and \ref{sec:a,b,a+b}.

\begin{lemma}\label{lem:multiple_lines}
For any $n\in \N$, choose some odd $b>4n+1$ and let $A=\{1,b,b+1\}$ and $S=(01^3)^n$. Then $\per(A,S)=2(n+1)b+1$ and $\preper(A,S)=0$.
\end{lemma}
A proof of this Lemma is found in Section \ref{sec:linesproof}. If we denote $b=4n+1+2d$, where $d\geq 1$, then the structure of the periodicity is exactly
\[\{w^{A,S}\}=\ldots \left(\sum_{i=0}^{n-1}\Big((11)^d(1^30)^i11(01^3)^{n-i}\ (01)^d0(1^30)^i11(01^3)^{n-i-1}0\Big)\ (11)^d(1^30)^n1^2\ (01)^d0(1^30)^n\right)^\infty,
\]
To choose a seed generating this structure it suffices to choose any sub-string of length $a+1$, so we choose the first $a+1$ entries, namely $(11)^d11(01^3)^n\approx (01^3)^n$. Lemma \ref{lem:multiple_lines} contradicts the generalization of Conjecture \ref{conj:linearbound} over all seeds, since it implies quadratic period lengths. Figure \ref{fig:allperiods1bbp1} plots all period lengths for $\{1,b,b+1\}$, which shows that this Lemma only scratches the surface of the periodicities of $3$-sets.
\begin{figure}
\centering
\includegraphics[scale=0.42]{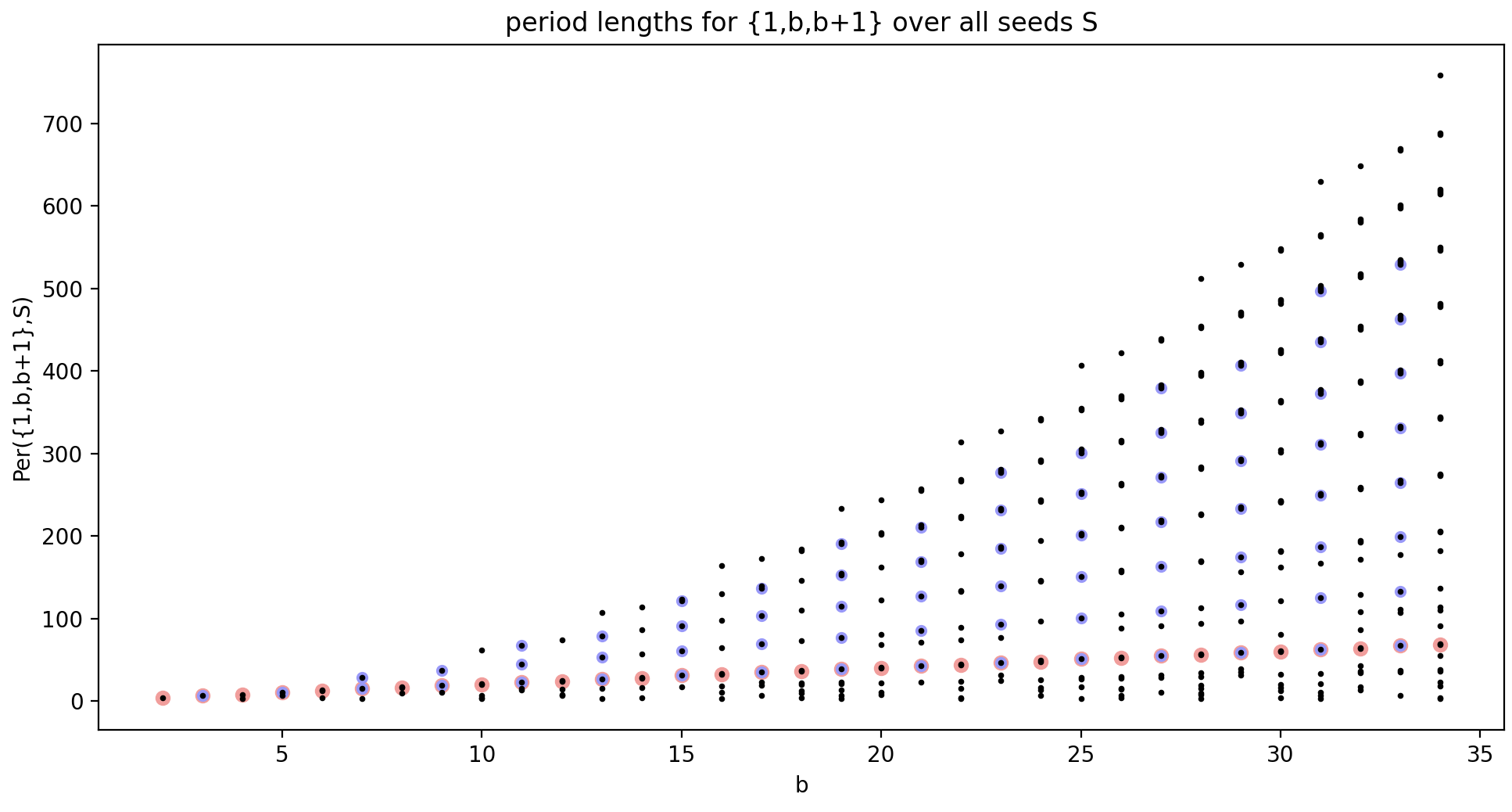}
\caption{This figure shows \textit{all} periods of $\{1,b,b+1\}$ for $b+1\leq 35$. The points highlighted in blue are given by Lemma \ref{lem:multiple_lines} and the points in red are the default periods. Note that some points close together are overlapping. For example, $\mathcal P^{\{1,31,32\}}=\{3, 7, 11, 21, 33, 63, 125, 167, 187, 249, 251, 311, 313, $\\$373, 375, 377, 435, 437, 439, 497, 499, 501, 503, 563, 565, 629\}$, while Lemma \ref{lem:multiple_lines} gives only $\{63, 125, 187,$\\$ 249, 311, 373, 435, 497\}$.}
\label{fig:allperiods1bbp1}
\end{figure}

\begin{theorem}[Superpolynomial Period Lengths]\label{thm:superpolynomial}  For $n\geq 1$, define $b_n=4n-1$, $A_n=\{n,nb_n,nb_n+n\}$, and $S_{(n)}=\sum_{j=1}^{n-1}0^j1^{b_n-j}$. For the family of pairs
$\{A_n,S_{(n)}\}_{n=1}^\infty$, where $\alpha_n=\max(A_n)$, it holds that
\[\per(A_n,S_{(n)})=e^{\Omega(\sqrt{\alpha_n})}.\]
\end{theorem}
\begin{proof}
Fix $n$, so we have $b=4n-1$ and $A=\{n,bn,bn+n\}$. We now construct the seed $S$ dependent on $n$. Because $A$ has greatest common divisor $n$, let $B=\{1,b,b+1\}$ and apply Multiplicative Linearity Proposition \ref{prop:linearity} to see that the parallel sequences satisfy $(w^{A,S})_i(m)=w^{A,S}(mn+i)=w^{B,S_i}(m)$ where $S_i\in \{0,1\}^{b+1}$ for $i\in \{0,\ldots,n-1\}$. We now apply Lemma \ref{lem:multiple_lines} by letting $S_i=(01^3)^i\approx 1^{4(n-i)}(01^3)^i$, so we have $\per(B,S_i)=2(i+1)b+1$. By combining all $S_i$ in parallel, this construction yields $S=\sum_{i=0}^{n-1}1^{n-i}0^i1^{3n}= \sum_{j=1}^n0^i1^{b-i}$. Now, suppose $\{w^{A,S}\}$ is periodic over some $p\in \N$. This implies that for all $m\in \N$ and $i\in \{0,\ldots, n-1\}$, we have
\[w^{B,S_i}(m)=w^{A,S}(mn+i)=w^{A,S}(mn+i+np)=w^{B,S_i}(m+p),
\]
so $\{w^{B,S_i}\}$ must also be periodic over $p$, i.e. $\per(B,S_i)\mid p$. Because this is true for all $i$, we conclude that $\lcm\{2(i+1)b+1\mid 0\leq i<n\}\leq \per(A_n,S)$.
A result from \cite[Problem 10797]{limitlcm} regarding the lcm of arithmetic progressions implies a somewhat loose bound of $\lcm(2b+1,4b+1,6b+1,\ldots,2na+1)\geq e^{n-o(n)}$.\footnote{In particular it says that $\lim_{n\to \infty }\frac{\log(\lcm(\ldots))}n=\frac{1}{\varphi(k)}\sum\frac{2b}{m}$ with the summation taken over units of $\Z_{2b}^\times$. This is an average of numbers greater than $1$.} Therefore for all $n$, we note that $\alpha_n\sim 4n^2$, and conclude $\per(A_n,S_n)=e^{\Omega(n)}=e^{\Omega(\sqrt{\alpha_n})}$. \end{proof}

Table \ref{tab:weaving} demonstrates the construction for $n=2$. This family proceeds as follows
\begin{itemize}
\item $n=1$, $b=3$, $A_1=\{1,3,4\}$, $S_1=\emptyset$, and $\per(A_1,S_1)=\lcm\{7\}=7$
\item $n=2$, $b=7$, $A_2=\{2,14,16\}$, $S_2=01^6$, $\per(A_2,S_2)=2\cdot \lcm\{15,29\}=870$
\item $n=3$, $b=11$, $A_3=\{3,33,36\}$, $S_3=01^{10}0^21^{9}$,  $\per(A_3,S_3)= 3\cdot \lcm\{23,45,67\}=208\,035$
\item $n=4$, $b=15$, $A_4=\{4,60,64\}$, $S_4=01^{14}0^21^{13}0^31^{12}$,  $\per(A_4,S_4)= 4\cdot \lcm\{31,61,91,121\} =83\,287\,204$
\item $n=5$, $b=19$, $A_5=\{5,95,100\}$, $S_5=01^{18}0^21^{17}0^31^{16}0^41^{15}$, and $3\,364\,005\,645 \mid \per(A_5,S_5)$; we have not computed the exact period.
\end{itemize}
It is natural to predict that we have a multiple of $n$ in the period length, but we have not proven this to be the case. To prove Theorem \ref{thm:superpolynomial} we used a rough bound on $\lcm\{2b+1,\ldots,2nb+1\}$ and the few periodicities from Lemma \ref{lem:multiple_lines}, as shown in Figure \ref{fig:allperiods1bbp1}. It is unclear if more periodicities and a stronger bound on their $\lcm$, would yield an asymptotically better result, but this does not appear to be the case. Figure \ref{fig:1bbp1bestcase} shows that $e^{\Omega(\sqrt{\alpha_n})}$ appears to be best possible.
\begin{figure}
\centering
\includegraphics[scale=0.4]{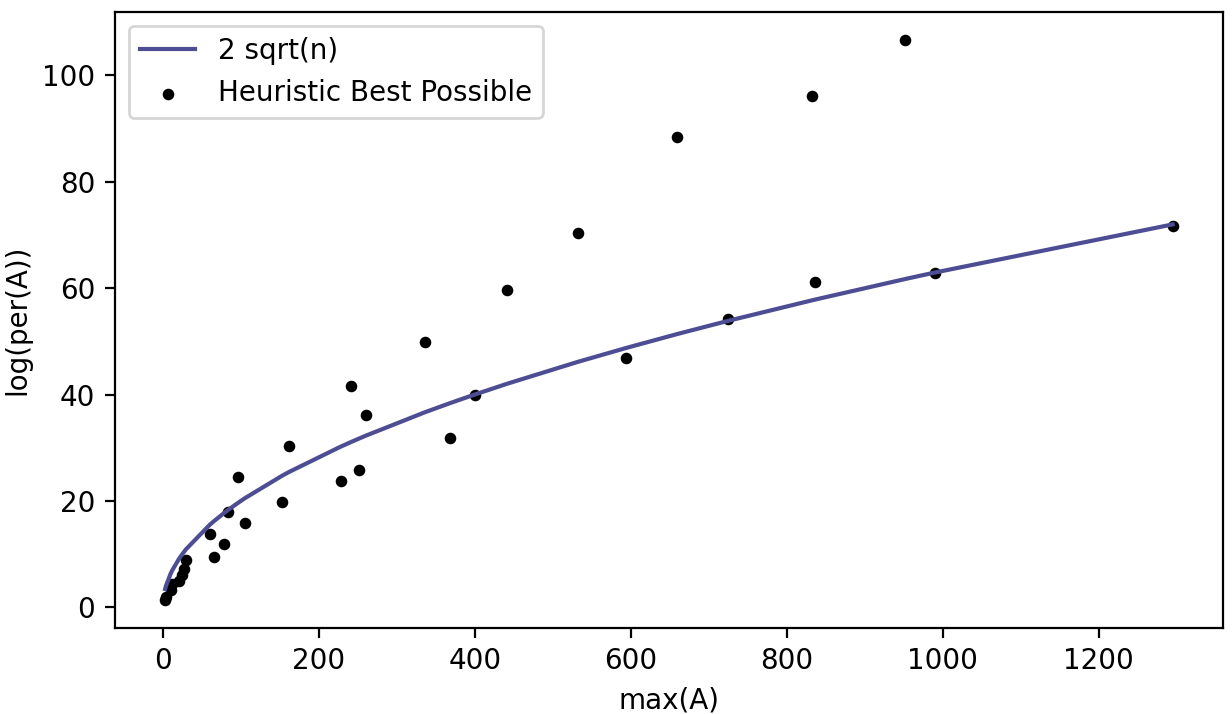}
\caption{Using the periods from Figure \ref{fig:allperiods1bbp1}, we can construct sets $A_b$ with many parallel $\{1,b,b+1\}$ periods, specifically we get $A_b=|\mathcal P^{\{1,b,b+1\}}|\cdot \{1,b,b+1\}$ and some seed $S$ comprised of $S_0,\ldots,S_{|\mathcal P^{\{1,b,b+1\}}|}$ such that $\per(A_b,S)\geq \lcm(\mathcal P^{\{1,b,b+1\}})$. This gives a heuristic for the longest period possible, which appears to approximate $e^{\bigO(\sqrt{ \max(A_b)})}$.}
 \label{fig:1bbp1bestcase}
\end{figure}

\begin{table}
    \centering
    \begin{tabular}{p{.05cm}p{.05cm}p{.05cm}p{.05cm}p{.05cm}p{.05cm}p{.05cm}p{.05cm}||p{.05cm}p{.05cm}p{.05cm}p{.05cm}p{.05cm}p{.05cm}p{.05cm}p{.05cm}p{.05cm}p{.05cm}p{.05cm}p{.05cm}p{.05cm}p{.05cm}p{.05cm}p{.05cm}p{.05cm}p{.05cm}p{.05cm}p{.05cm}p{.05cm}p{.05cm}p{.05cm}p{.05cm}p{.05cm}p{.05cm}p{.05cm}p{.05cm}p{.05cm}p{.05cm}}
1&&1&&1&&1&&0&&1&&0&&1&&0&&1&&0&&1&&\ldots \\
&0&&1&&1&&1&&0&&1&&0&&1&&1&&0&&1&&1&\ldots
\end{tabular}
\caption{In the construction for Theorem \ref{thm:superpolynomial}, we have $A_2=\{2,14,16\}$. The sequences $(w^{A,S})_i(m)$ for $i\in \{0,1\}$ are independent, and these have periods of lengths $2\cdot 7+1=15$ and $4\cdot 7+1=29$. We find that $\per(A_2,01^6)) =2\cdot\lcm(15,29)=870$.}
    \label{tab:weaving}
\end{table}
\subsection{Proof of Lemma \ref{lem:multiple_lines}.}\label{sec:linesproof}
The final proof uses different approach to the previous ones. Rather than constructing the sequence from scratch, we will exploit a structural pattern to \textit{extend} a simple period. This was the method used to discover Lemma \ref{lem:multiple_lines}, and might be a productive strategy moving forward.

\begin{proof}
First, we note that if $n=0$, then Theorems \ref{thm:1bc} and \ref{thm:aba+b} coincide to give that $\per(A)=2b+1$ with no seed, so we are done. Next, fix $n\geq 1$.

Observe that if $\idk=4n+1$ and $B=\{1,\idk,\idk+1\}$ then the string $X=110(1^30)^n$ is a valid period structure of $B$ with period length $\idk+2$. This follows from the fact that $X^\infty(n)=1$ if and only if $X^\infty(n-1)=0$, $X^\infty(n+1)=0$, or $X^\infty(n+2)=0$, meaning each zero is separated by $2$ or $3$ ones. We use this fact to explicitly build period structures for odd $b>\idk$.

Write $2n+1$ copies of $X$ in a $(2n+2)\times (\idk+1)$ grid so that the rows $0,\ldots,2n+1$ are the following: The $k\nth$ even row indexed from zero reads $(1^30)^k11 (01^3)^{n-k}$ and the $k\nth$ odd row reads $0(1^30)^k11(01^3)^{n-k-1}0$, except for the last row, (the $n\nth$ odd row), and this reads $0(1^30)^{n}$. Table \ref{tab:examplegridfilling} gives two examples of such a filling. Notice that this filling has the following three properties.
\begin{itemize}
\item Reading across all of the rows yields $X^{2n+1}$ in order.
\item Even rows are fully filled. These begin and end with strings $1^3$ or $11$. Their length is $\idk+1=|X|-1$ so they contain all of $X$ except for a $0$.
\item Odd rows have $\idk-1$ entries, except for row $n$ which has $\idk$ entries. These begin and end with $0$, and contain all of $X$ except either $1^3$ or $11$.
\end{itemize}
\begin{table}[H]
\centering
\begin{tabular}{|c|c|c|c|c|c|}
\hline
1&1&0&1&1&1  \\
\hline
0&1&1&0&&  \\
\hline
1&1&1&0&1&1  \\
\hline
0&1&1&1&0&  \\
\hline
\end{tabular}
\quad
\begin{tabular}{|c|c|c|c|c|c|c|c|c|c|}
\hline
\textbf{1}&\textbf{1}&\textbf{0}&\textbf{1}
&\textbf{1}&\textbf{1}&\textbf{0}&\textbf{1}
&\textbf{1}&\textbf{1} \\
\hline
\textbf{0}&\textit{1}&\textit{1}&\textit{0}&\textit{1}&\textit{1}&\textit{1}&\textit{0}&& \\
\hline
\textit{1}&\textit{1}&\textit{1}&\textit{0}&\textbf{1}&\textbf{1}&\textbf{0}& \textbf{1}&\textbf{1}&\textbf{1} \\
\hline
\textbf{0}&\textbf{1}&\textbf{1}&\textbf{1}
&\textbf{0}&\textit{1}&\textit{1}&\textit{0}&& \\
\hline
\textit{1}&\textit{1}&\textit{1}&\textit{0}&\textit{1}&\textit{1}&\textit{1}&\textit{0}&\textbf{1}&\textbf{1} \\
\hline
\textbf{0}&\textbf{1}&\textbf{1}&\textbf{1}&\textbf{0}&\textbf{1}&\textbf{1}&\textbf{1}&\textbf{0}& \\
\hline
\end{tabular}
\caption{\textbf{Left.} Setting $n=1$, we get $B=\{1,5,6\}$. We place $2n+1=3$ copies of $X=1101^30$ into a grid, so odd rows start and end with $0$. \textbf{Right.} Setting $n=2$, we get $B=\{1,9,10\}$. We place $5$ copies of $X=1101^301^30$ into a grid. Alternating copies of $X$ are bolded to illustrate the pattern.}
\label{tab:examplegridfilling}
\end{table}
Denote $x(i,j)$ as the element in row $i$ and column $j$, both indexed from $0$. We can now re-interpret the recurrence relation on $w^B(n)$ in terms of $x$. The following identities must hold for all $i,j$.
\begin{equation}\label{eq:tablerecurrence}
x(i,j)=\begin{cases}
1-\min\{x(i,j-1),x(i-1,j),x(i-1,j+1)\}&i\text{ is odd and }j>0\\
1-\min\{x(i,j-1),x(i-1,j-2),x(i-1,j-1)\}&i>0 \text{ is even},\idk>j>1\\
\\
1-\min\{x(i-1,\idk),x(i-1,0),x(i-1,1)\}&i\text{ is odd and }j=0
\\
1-\min\{x(0,j-1),x(2n+1,j),x(2n+1,j-1)\}&i=0,\idk>j>0
\\
1-\min\{x(2n+1,\idk-1),x(2n+1,0),x(2n,\idk)\}&i=0,j=0
\\
1-\min\{x(0,\idk-1),x(0,0),x(2n+1,\idk-1)\}&i=0,j=\idk
\\
1-\min\{x(i,\idk-1),x(i,0),x(i-1,\idk-2)\}&i>0 \text{ is even},j=\idk
\\
1-\min\{x(i,0),x(i-1,0),x(i-2,\idk)\}&i>0 \text{ is even},j=1
\\
1-\min\{x(i-1,\idk-2),x(i-2,\idk),x(i-2,\idk-1)\}&i>0 \text{ is even},j=0
\end{cases}
\end{equation}
We focus on the first two cases which are the most general. It might also help to recall that for odd $i$, $x(i,0)=0$ and for even $i$, $x(i,0)=x(i,1)=x(i,\idk-1)=x(i,\idk)=1$. These follow from the specifications of the grid filling.

Next, we will choose any odd $b=\idk+2d$ for $d>0$, and let $A=\{1,b,b+1\}$. We will extend the grid to create a period structure for $A$ with length $2(n+1)b+1$. We add $d$ copies of the first two columns on the left side, as shown below.
\begin{table}[H]
\centering
\begin{tabular}{|c|c||c|c||c|c||c|c|c|c|c|c|}
\hline
1&1&1&1&1&1& 1&1&0&1&1&1  \\
\hline
0&1&0&1&0&1& 0&1&1&0&&  \\
\hline
1&1&1&1&1&1& 1&1&1&0&1&1  \\
\hline
0&1&0&1&0&1& 0&1&1&1&0&  \\
\hline
\end{tabular}
\ =\ 
\begin{tabular}{|c|c|c|c|c|}
\hline
(11)$^{d+1}$&0&1&1&1  \\
\hline
(01)$^{d+1}$&1&0&&  \\
\hline
(11)$^{d+1}$&1&0&1&1  \\
\hline
(01)$^{d+1}$&1&1&0&  \\
\hline
\end{tabular}
\quad
\caption{Example of \textit{extending} the grid for $b=\idk+2d$ with $n=1$ and $d=3$.}
\label{tab:examplegridfilling2}
\end{table}
We call the elements of this modified grid $y(i,j)$, with $i$ indexed from $0$ and $j$ indexed from $-2d$. Thus for $j\geq 0$ we have $y(i,j)=x(i,j)$. If $j<0$ is even we have $y(i,j)=x(i,0)$ and if $j<0$ is odd we have $y(i,j)=x(i,1)$. Because of the specifications for filling $x$, we note that this entails prepending $(11)^d$ to even rows and $(01)^d$ to odd rows.

\textbf{Claim.} \textit{Reading across the rows of this table yields $Y$, a valid period structure for $A$ with length ${2(n+1)b+1}$.}
\vskip 3pt
\noindent First we have added $(2n+2)2d$ entries to the table, and we started with $X^{2n+1}$, so the total length is 
\[
(2n+2)2d+(2n+1)(\idk+2)
=(2n+2)2d+(2n+2)\idk+1
=2b(n+1)+1
\]
Next we show that the recurrence relation is satisfied. To do this, write the recurrence relation $w^A(n)=1-\min\{w^A(n-1),w^A(n-\idk-2d),w^A(n-\idk-1-2d)$ in terms of $y(i,j)$. The rows are extended by the same length as the recurrence, so the following rules very are similar to Equation \ref{eq:tablerecurrence}.

\begin{equation}\label{eq:tablerecurrencey}
y(i,j)=\begin{cases}
1-\min\{y(i,j-1),y(i-1,j),y(i-1,j+1)\}&i\text{ is odd and }j>-2d\\
1-\min\{y(i,j-1),y(i-1,j-2),y(i-1,j-1)\}&i>0 \text{ is even},\idk>j>1-2d\\
\\
1-\min\{y(i-1,\idk),y(i-1,0),y(i-1,1)\}&i\text{ is odd and }j=0
\\
1-\min\{y(0,j-1),y(2n+1,j),y(2n+1,j-1)\}&i=0,\ \idk>j>-2d
\\
1-\min\{y(2n+1,\idk-1),y(2n+1,\idk),y(2n+1,-2d)\}&i=0,j=-2d
\\
1-\min\{y(0,\idk-1),y(0,-2d),y(2n+1,\idk-1)\}&i=0,j=\idk
\\
1-\min\{y(i,\idk-1),y(i,-2d),y(i-1,\idk-2)\}&i>0 \text{ is even},j=\idk
\\
1-\min\{y(i,-2d),y(i-1,-2d),y(i-2,\idk)\}&i>0 \text{ is even},j=-2d+1
\\
1-\min\{y(i-1,\idk-2),y(i-2,\idk),y(i-2,\idk-1)\}&i>0 \text{ is even},j=-2d
\end{cases}
\end{equation}
From here it is possible to check that the nine cases in Equation \ref{eq:tablerecurrencey} hold for all $j\geq0$ and $j<0$ using Equation \ref{eq:tablerecurrence}, the definition of $y(i,j)$ and the three properties of the filling. We will show only the two main cases; the rest follow suit.

If $i$ is odd and $0<j<\idk$, then we confirm
\begin{multline*}\label{eq:exrecur1}
y(i,j)=x(i,j)=1-\min\{x(i,j-1),x(i-1,j),x(i-1,j+1)\}
\\=1-\min\{y(i,j-1),y(i-1,j),y(i-1,j+1)\}.
\end{multline*}
If $i$ is odd and $-2d<j\leq 0$ is even, then we know $y(i,j)=x(i,0)=0$, so we confirm
\[
y(i,j)=1-\min\{y(i,j-1),y(i-1,j),y(i-1,j+1)\}=1-\min\{1,1,1\}=0.
\]
If $i$ is odd and $-2d<j<0$ is odd, then we know $y(i,j)=x(i,1)=1$, so we confirm
\[
y(i,j)=1-\min\{y(i,j-1),y(i-1,j),y(i-1,j+1)\}=1-\min\{0,1,0\}=1.
\]
If $i>0$ is even and $1<j<\idk$, then 
\begin{multline*}\label{eq:exrecur2}
y(i,j)=x(i,j)=1-\min\{x(i,j-1),x(i-1,j-1),x(i-1,j-2)\}\\=1-\min\{y(i,j-1),y(i-1,j-1),y(i-1,j-2)\}.
\end{multline*}
If $i>0$ is even and $-2d+1<j\leq 1$ is odd, then we know $y(i,j)=x(i,1)=1$, so we confirm
\[
y(i,j)=1-\min\{y(i,j-1),y(i-1,j-1),y(i-1,j-2)\}
=1-\min\{1,0,1\}=1.
\]
If $i>0$ is even and $-2d+1<j\leq 0$ is even, then we know $y(i,j)=x(i,0)=1$, so we confirm
\[
y(i,j)=1-\min\{y(i,j-1),y(i-1,j-1),y(i-1,j-2)\}
=1-\min\{1,1,0\}=1.
\]
Such verification can be completed for the seven edge cases to show that $y(i,j)$ obeys Equation \ref{eq:tablerecurrencey}, and therefore $Y^\infty\in \mathcal W^{A}$.

Finally, we check for subperiods of $Y$. Note that as long as $d>0$, there are exactly $(n+1)$ copies of the substring $(01)^{d+1}$ present in $Y$, each at the beginning of a row. Additionally, the last row uniquely has length $b$, meaning the instances of $(01)^{d+1}$ are unequally spaced throughout $Y$. This prevents a sub-period of $Y$ from occurring.
\end{proof}
\section{Closing}
We close the paper by setting up a few conjectures. First, we recall the observation that the converse of Proposition \ref{prop:invariant-extension} holds for all $(A,S)$ where $|A|\leq 3$, and state it explicitly in a conjecture.
\begin{conjecture} If $|A|\leq 3$, then for all $S\in\{0,1\}^\alpha$, let $p=\per(A,S)$. Then if $kp+x$ is an extension of $A$ for all $k\in \N_+$ and $x\in A$, then $\preper(A,S)=0$.
\end{conjecture}
It suffices to check the case $k=1$. If $S=\emptyset$ this is precisely the converse of Proposition \ref{prop:invariant-extension}; otherwise the converse is not true.

The following conjecture is based on computer simulations of $A\in {[200]\choose 3}$. It provides a very limited characterization of the general $\{a,b,c\}$ case but provides an example of its complex behavior.

\begin{conjecture}\label{conj:abcperbc}
Suppose $A=\{a,b,c\}$ with $1<a<b<c$. Use the division algorithm to obtain the following variables:
\begin{align*}
b&=q\ a+r\\
c&=q_c(a+b)+r_c&
r_c&=q_a(2a)+r_a\\
c-a&=q_c'(a+b)+r_c'&
r_c'&=q_a'(2a)+r_a'
\end{align*}
Then $\per(A)=b+c$ and $\preper(A)=0$ if and only if one of the following holds:
\begin{enumerate}[(i)]
\item $q$ is even and $r_c'>0$, $r_a'\leq r$, and $2q_a'\leq q$, and if $2q_a'=q$ then $r_a'\leq 2r-a$.
\item $q$ is odd and $r\neq 0$, $r\leq r_a\leq a$, and if $q_a=0$ then $r_a<a$.
\item $q$ is odd and $r=0$, and $r_a\neq a$
\end{enumerate}
\end{conjecture}

The following conjecture has been verified for all $A\in {[25]\choose 3}$, for all seeds $S$. Together with Conjecture \ref{conj:linearbound} these are a strengthening of a conjecture by Alth\"ofer and B\"ultermann in \cite[(i)]{superlinear}.
\begin{conjecture}\label{conj:quadraticbound}
If $A=\{a,b,c\}$ with $a<b<c$, and if  $\gcd(a,b,c)=1$, then for all seeds $S\in \{0,1\}^\alpha$, we have $\per(A,S)<c^2$.
\end{conjecture}

This would imply that the construction in Theorem \ref{thm:superpolynomial} can only occur if $A$ has a greatest common divisor. It would also provide the general bound that if $g=\gcd(a,b,c)$ and $\tilde c=c/g$, then for all seeds $S$, $\per(A,S)\leq \tilde c^{2g}$, which is $\bigO(e^{2c/e})$ in the worst case but ordinarily much less than $ (\min(a+b,c)+1)2^{c-3}$, the bound derived in Theorem \ref{thm:periodbound}.

From computer simulations of $A\in {[25]\choose 3}$ with $\gcd A=1$, we note that for all seeds $S$, if the period is super-linear, i.e. $\per(A,S)>2c$, then $(a+b)\mid c$, with the following eight exceptions. These sets have $\max(\mathcal P^A)>2\alpha$, with example seeds given.
\begin{center}
\begin{tabular}{lr}
$\per(\{11, 16, 20\},(01)^21^401^2)$&$=61$\\
$\per(\{3, 11, 21\},0^2101^20)$&$=61$\\
$\per(\{7, 17, 23\},(010)^201^2)$&$=73$\\
$\per(\{10, 21, 23\},0^2101^4)$&$=78$\\
$\per(\{5, 11, 24\},01^20^31^2010)$&$=65$\\
$\per(\{11, 16, 25\},01^2010^21^4)$&$=56$\\
$\per(\{16, 21, 25\}, 0^2(101^3)^201^2)$&$=61$\\
$\per(\{13, 23, 25\},0^21^2010^21^5)$&$=83$
\end{tabular}
\end{center}
A characterization of these sets may relate to the solution to Conjecture \ref{conj:linearbound} and/or \ref{conj:quadraticbound}.

\section*{Acknowledgement}
IM was supported by NKFIH grant K132696. The project is a continuation of the work done at the 2022 Spring semester at Budapest Semesters in Mathematics. Both authors thank to Mariam Wael Abu-Adas for participating in the very early phase of the research work at the BSM leading to the results presented in this paper.
Both authors would like to thank the BSM for running the program.

\newpage
\bibliography{main}{}
\bibliographystyle{plain}

\end{document}